\documentclass[11pt,reqno]{amsart}
\usepackage[latin1]{inputenc}
\usepackage{amsmath}
\usepackage{amsfonts}
\usepackage{amssymb}
\usepackage{graphicx}
\usepackage{plain}
\usepackage{eurosym}
\usepackage{url}

\usepackage{mathrsfs}

\usepackage[english]{babel}                       
\usepackage{latexsym}                            

\usepackage{amsthm}
\usepackage{amssymb}
\usepackage{longtable}                           

\usepackage{exscale}

\usepackage{pst-plot}                            

\usepackage{array}
\newcolumntype{I}{!{\vrule width 2pt}}

\usepackage{bbm}

\usepackage{listing}
\usepackage{color}

\definecolor{dkgreen}{rgb}{0,0.6,0}
\definecolor{gray}{rgb}{0.5,0.5,0.5}
\definecolor{mauve}{rgb}{0.58,0,0.82}

\theoremstyle{definition}
\newtheorem{lem}{Lemma}[section]
\newtheorem{thm}{Theorem}[section]
\newtheorem{dfn}[lem]{Definition}

\usepackage{amssymb}
\usepackage{stmaryrd}

\usepackage{tikz}
\usetikzlibrary{calc}

\usepackage[all,cmtip]{xy}

\usepackage{amsthm}

\usepackage{fdsymbol}

\usepackage{ stmaryrd }

\usepackage{csquotes}

\usepackage{hyperref}

\usepackage{tikz}

\newcommand{\bbR}{\mathbb{R}_{\geq 0}}

\newcommand{\Ouv}{\mathrm{Ouv}}

\newcommand{\oversum}[1]{\overset{#1}{\sum}}

\newcommand{\depsum}[2]{\underset{#1}{\overset{#2}{\sum}}}

\newcommand{\laxcolim}[1]{\underset{#1}{\mathrm{colim}^{lax}}  }
\newcommand{\oplaxcolim}[1]{\underset{#1}{\mathrm{colim}^{oplax}}  }

	\DeclareMathSymbol{\mlq}{\mathord}{operators}{'134}
	\DeclareMathSymbol{\mrq}{\mathord}{operators}{'42}
	
\begin{document}
	
\footskip30pt

\baselineskip=1.1\baselineskip

\title{Towards a General Theory of Dependent Sums}

\author{Peter Bonart}

\begin{abstract}
We introduce dependent adders. A dependent adder $A$ has for every $x \in A$ a way of adding together $x$ many elements of $A$. We provide examples from many disparate branches of mathematics. Examples include the field with one element $\mathbb{F}_1$, the real numbers with integrals as sums, the category of categories with oplax colimits as sums. We also consider modules over dependent adders and provide examples.
\end{abstract}

\maketitle
\thispagestyle{empty}
\pagestyle{plain}

\tableofcontents

\section{Introduction}

This paper defines \textit{dependent adders} and constructs examples for them.
Before precisely defining them, we define them roughly and enumerate the examples we are interested in.
Defined roughly, a \textit{dependent adder} consists of a set $A$, an element $1 \in A$, for every $x \in A$ a set $A^x$, called the \textit{set of $x$-indexed families} in $A$, and a function $\depsum{i}{x} : A^x \rightarrow A$ called the \textit{$x$-dependent sum function}, satisfying nice properties, like that $\depsum{i}{x} 1 = x$ and that every sum of the form $\depsum{j}{\depsum{i}{x}f(i)}g(j)$ can be rewritten as a sum of the form $\depsum{i}{x} \depsum{j}{f(i)} g(\phi(i,j))$ for some well-behaved function $\phi(i,j)$ depending only on $f$ and not on $g$.

Examples of dependent adders include:
\begin{enumerate}
	\item The set of natural numbers $\mathbb{N}$. For each $x \in \mathbb{N}$, an $x$-indexed family of natural numbers is a function $f : \{1,\dots,x\} \rightarrow \mathbb{N}$, and its sum is the usual sum $$\depsum{i}{x} f(i) := \underset{i=1}{\overset{x}{\sum}} f(i)$$
	\item The set of non-negative real numbers $\bbR$. For each $x \in \bbR$, we define an $x$-indexed family of real numbers to be a continuous function $f: [0,x] \rightarrow \bbR$.
	The sum of such a function is defined to be the integral
	$$\depsum{i}{x} f(i) := \underset{0}{\overset{x}{\int}} f(t)dt $$
	This integral always exists, is non-negative and finite.
	\item The category of small categories $Cat$.
	Define an $I$-indexed family of categories to be an isomorphism class of functors $F: I^{op} \rightarrow Cat$.
	The sum of $F$ is defined to be the lax $2$-colimit of $F$, also known as the Grothendieck construction
	$$\depsum{i}{I}F(i) := \underset{i \in I^{op}}{\mathrm{colim}}^{lax} F(i) = \int F$$
	Alternatively, one can define an $I$-indexed family of categories to be an isomorphism class of functors $F: I \rightarrow Cat$ and define its sum to be the oplax colimit of $F$.
	We need to take isomorphism classes of functors instead of just taking functors, because otherwise $Cat$ would only satisfy our axioms up to isomorphisms, instead of satisfying them strictly.
	\item The $p$-adic integers $\mathbb{Z}_p$.
	For every $x \in \mathbb{Z}_p$ we define an $x$-indexed family to be a continuous function $f : \mathbb{Z}_p \rightarrow \mathbb{Z}_p$.
	Since $\mathbb{N}$ is dense in $\mathbb{Z}_p$ we can find a sequence of natural numbers $x_n \in \mathbb{N}$ such that $x_n \rightarrow x$ in $\mathbb{Z}_p$.
	We define $$\depsum{i}{x} f(i) := \underset{n \rightarrow \infty}{\mathrm{lim}} \underset{i=1}{\overset{x_n}{\sum}} f(i) $$
	It is an easy exercise in $p$-adic analysis to show that this limit always exists and is independent of the choice of sequence $x_n$.	(See Lemma \ref{lemmaZpsums} for a proof of this claim).
	\item The set of all small cardinals $Card$, where a $\kappa$-indexed family is a function $\kappa \rightarrow Card$, and where the sums are the usual infinite sums of cardinals.
	\item The set of all small ordinal numbers $Ord$.
	For each ordinal $\alpha$ an $\alpha$-indexed family is a function of sets $f : \alpha \rightarrow Ord$.
	The sum of $f$ is defined via transfinite recusion on $\alpha$:
	$$\hspace{30pt} \depsum{x}{0} f(x) := 0 \quad \depsum{x}{\alpha +1} f(x) := (\depsum{x}{\alpha} f(x)) + f(\alpha) \quad \depsum{x}{\underset{\beta < \alpha}{\sup} \beta} f(x) := \underset{\beta < \alpha}{\sup} \depsum{x}{\beta} f(x) $$
	\item The two element set $\mathbb{F}_1 := \{0,1\}$, which is also sometimes called \enquote{the field with one element}\cite{lorscheid2018F1}. We define a $0$-indexed family to be a function from the empty set $\emptyset \rightarrow \mathbb{F}_1$, i.e. there is exactly one $0$-indexed family, which we call the empty family and denote  $\emptyset$.
	We define a $1$-indexed family to be a function $f: \{*\} \rightarrow \mathbb{F}_1$, i.e. there are exactly two $1$-indexed families $0$ and $1$ corresponding to the elements $\mathbb{F}_1$.
	We define $$\depsum{i}{0} \emptyset ;= 0 \quad \depsum{i}{1} 0 := 0 \quad \depsum{i}{1} 1 := 1$$
	\item The real unit interval $[0,1]$. Just as for the real numbers, for each $x \in [0,1]$ we define an $x$-indexed family to be a continuous function $f: [0,x] \rightarrow [0,1]$, and define the sum to be the integral
	$$ \depsum{i}{x} f(i) := \underset{0}{\overset{x}{\int}} f(t)dt$$
	This integral is again in $[0,1]$, because for every $t$ we have $f(t) \leq 1$, so $$\underset{0}{\overset{x}{\int}} f(t)dt \leq \underset{0}{\overset{x}{\int}} 1 dt = t \leq 1 $$
	\item The set of integers $\mathbb{Z}$. For each $x \in \mathbb{Z}$, an $x$-indexed family of integers is a function $f : \mathbb{Z} \rightarrow \mathbb{Z}$. For $x \geq 0$ the sum of $f$ is the usual sum. For $x < 0$ we define
	$$ \depsum{i}{x} f(i) := \underset{i=1}{\overset{-x}{\sum}} -f(1-i) $$
	\item The set of all (possibly negative) real numbers $\mathbb{R}$ with integrals as sums, where an $x$-indexed family is a continuous function $\mathbb{R} \rightarrow \mathbb{R}$.
	\item The interval $[-1,1]$ with integrals as sums, where an $x$-indexed family is a continuous function $[-1,1] \rightarrow [-1,1]$.
	\item The complex numbers $\mathbb{C}$.
	For every $z \in \mathbb{C}$ we define a $z$-indexed family to be an entire holomorphic function $f : \mathbb{C} \rightarrow \mathbb{C}$.
	We can choose a smooth path $\gamma : [0,1] \rightarrow \mathbb{C}$ that leads from $0$ to $z$, i.e. $\gamma(0) =0$ and $\gamma(1)=z$.
	We define
	$$\depsum{i}{z} f(i) :=  \underset{\gamma}{\int} f(t)dt = \underset{0}{\overset{1}{\int}} f(\gamma(t)) \gamma^\prime(t) dt $$
	By Cauchy's integral theorem this is independent of the choice of $\gamma$.
	\item Every commutative $\mathbb{Q}$-algebra $R$. For every $r \in R$ we define an $r$-indexed family to be a polynomial function $p: R \rightarrow R$, i.e. an element of the polynomial ring $p \in R[X]$.
	For every $d \in \mathbb{N}$ we define the $d$-th Faulhaber polynomial $F_d \in R[X]$ by
	$$F_d := \frac{1}{d+1}\depsum{n=0}{d} \binom{d+1}{n}B_nX^{d-n+1}$$where $B_n$ is the Bernoulli number with convention $B_1 = +\frac{1}{2}$.
	For $p \in R[X]$ we can write $p$ in the form $p = \depsum{i=0}{n}p_iX^i$ and define for every $x \in R$ the $x$-dependent sum of $p$ by
	$$\depsum{i}{x}p(i) := \depsum{i=0}{n}p_i \cdot F_i(x)$$
	Then $\depsum{i}{x}p(i)$ is a polynomial function in $x$ and for every $m \in \mathbb{N}$ the $m$-dependent sum $\depsum{i}{m}p(i)$ is equal to the usual sum $\depsum{i=1}{m}p(i)$, because of Faulhaber's formula $F_d(m) = \depsum{i=1}{m} i^d$. \cite{orosi2018simple}
	
\end{enumerate}

This paper tries to capture what all of these examples have in common in a unified framework.
In Section \ref{sectionDependentAdders} we define dependent adders, and very briefly investigate their most rudimentary properties. However we do not develop this theory very far in this paper. The majority of this paper is of a zoological nature and consists of definitions, examples, and proofs that the examples fit the definition.
In Section \ref{sectionExamples} we show that all the examples we listed above are in fact dependent adders.
In Section \ref{sectionRightModules} and \ref{sectionLeftModules} we consider right and left modules over dependent adders.
Given a dependent adder $A$, a right $A$-module $M$, is a place where one can form sums of the form $\depsum{i}{x} f(i) \in M$ where $x \in A$, $f(i) \in M$.
$\mathbb{F}_1$-modules are exactly the pointed sets.
$\mathbb{N}$-modules are exactly the monoids.
Every Banach space with its Bochner integrals is an $\mathbb{R}$-module.
Every cocomplete category is a $Cat$-module.
A left $A$-module $M$ is a place where one can form sums of the form $\depsum{i}{m} f(i)\in M$ where $m \in M$ and $f(i) \in A$. We show that the category of topological spaces an open maps $Top_{open}$ is a left module over the category of $Sets$, if one defines for $X \in Top_{open}$ an $X$-indexed family of sets to be a presheaf $\mathscr{F}$ on $X$, and defines the sum of $\mathscr{F}$ to be the étalé space of $\mathscr{F}$.

Our investigations have mostly been inspired by $\mathbb{F}_1$-geometry \cite{lorscheid2018F1} and Arakelov theory \cite{durov2007new}, and then derailed into something more general.
The \enquote{field with one element} $\mathbb{F}_1$ and the \enquote{complete local ring at infinity} $[-1,1]$ might not have binary sums $x + y$, but they have all dependent sums $\depsum{i}{x}f(i)$.

\section{Dependent Adders}\label{sectionDependentAdders}
Let $\mathscr{C}$ be a category with finite limits.
We denote the terminal object of $\mathscr{C}$ by $*$.

\begin{dfn}
	A \textit{dependent adder} in $\mathscr{C}$ consists of:
	\begin{enumerate}
		\item An object $A \in \mathscr{C}$.
		\item A morphism $p : F \rightarrow A$ in $\mathscr{C}$. We informally think of it as being an $A$-indexed family of objects of $\mathscr{C}$. For any morphism $x : U \rightarrow A$ we write $\llbracket x \rrbracket$ for the pullback of the diagram $U \overset{x}{\rightarrow} A \overset{p}{\leftarrow} F$ and write $p_x$ for the canonical map $\llbracket x \rrbracket \rightarrow U$.
		
		We informally think of a morphism $x : U \rightarrow A$ to be a generalized element of $A$, and a morphism $\llbracket x \rrbracket \rightarrow A$ to be an $x$-indexed family of elements of $A$.
		\item A morphism $1_F : * \rightarrow F$, called the \textit{unit}. We also define $1_A :* \rightarrow A$ by $1_A := p \circ 1_F$.
		\item For every $x : U \rightarrow A$ in $\mathscr{C}$ a function of sets
		$$\oversum{x} : \mathrm{Hom}_{\mathscr{C}}(\llbracket x \rrbracket, A) \rightarrow \mathrm{Hom}_{\mathscr{C}}(U,A)$$
		called the \textit{$x$-dependent sum function}.
		This map is supposed to be natural in $x \in \mathrm{Ob}(\mathscr{C}/A)$ in the sense that for every triangle of the form
		$$\xymatrix{ U \ar[rr]^{f} \ar[dr]_x & & U^\prime \ar[dl]^{y}\\
		             & A & }$$
	    the following diagram commutes
	    $$\xymatrix{ \mathrm{Hom}_{\mathscr{C}}(\llbracket x \rrbracket, A) \ar[r]^(0.55){\overset{x}{\sum}} & \mathrm{Hom}_{\mathscr{C}}(U,A)  \\
	     \mathrm{Hom}_{\mathscr{C}}(\llbracket y \rrbracket, A) \ar[r]^(0.55){\overset{y}{\sum}}  \ar[u]_{(f \underset{A}{\times} id_F)^*}& \mathrm{Hom}_{\mathscr{C}}(U^\prime,A) \ar[u]^{f^*} } $$
        \item For every $x : U \rightarrow A$ and every $f : \llbracket x \rrbracket \rightarrow A$ 
        a morphism
        $$f^\flat : \llbracket f \rrbracket \rightarrow \llbracket \oversum{x}f \rrbracket $$
        called the \textit{flattening map of $f$}. This is supposed to be a morphism in $\mathscr{C}/U$ in the sense that the following diagram commutes:
        $$\xymatrix{ \llbracket f \rrbracket \ar[d]_{p_f} \ar[r]^{f^\flat} & \llbracket \oversum{x} f \rrbracket \ar[d]^{p_{(\oversum{x} f)}} \\
        \llbracket x \rrbracket \ar[r]^{p_x} & U } $$
	\end{enumerate}
	We now demand that the following axioms are satisfied:
	\begin{enumerate}
		\item Right Unit axiom: For every $x : U \rightarrow A$ let $const_{1_A} : \llbracket x \rrbracket \rightarrow A$ be the composite function $\llbracket x \rrbracket \rightarrow * \overset{1_A}{\rightarrow} A$.
		We demand that $$\oversum{x} const_{1_A} = x$$
		\item Left Unit axiom: For every $x : U \rightarrow A$ and $f : \llbracket x \rrbracket \rightarrow A$, let $const_{1_A}: \llbracket x \rrbracket \rightarrow A$ be the constant $1_A$ function. We demand that
		$$\oversum{const_{1_A}} f \circ const_{1_A}^\flat = f$$
		\item Sum Associativity Axiom: For every $x : U \rightarrow A$, $f : \llbracket x \rrbracket \rightarrow A$ and $g : \llbracket \oversum{x}f \rrbracket \rightarrow A$ we demand that
		$$\oversum{\oversum{x} f} g = \oversum{x} \oversum{f} g \circ f^\flat  $$
		\item Flatten Associativity Axiom: For every $x : U \rightarrow A$, $f : \llbracket x \rrbracket \rightarrow A$, $g : \llbracket \oversum{x}f \rrbracket \rightarrow A$ the following diagram commutes
		$$\xymatrix{
			\llbracket g \circ f^\flat \rrbracket \ar[d]_{ f^\flat \underset{A}{\times} id_F  } \ar[r]^{( g \circ f^\flat)^\flat} & \llbracket \oversum{f} g \circ f^\flat \rrbracket \ar[d]^{( \oversum{f} g \circ f^\flat)^\flat  } \\
			\llbracket g \rrbracket \ar[r]^(.45){g^\flat} & \llbracket \oversum{x} \oversum{f} g \circ f^\flat \rrbracket 
		} $$
	\end{enumerate}
	
\end{dfn}

\begin{dfn}
	A dependent adder $A$ in $\mathscr{C}$ is called \textit{commutative} if it satisfies the following Fubini axiom: For every $x : U \rightarrow A$, $y : U \rightarrow A$ and $f : \llbracket x \circ p_y \rrbracket \rightarrow A$, we have a canonical isomorphism $switch : \llbracket x \circ p_y \rrbracket \rightarrow \llbracket y \circ p_x \rrbracket$ and demand that
	$$\oversum{y} \oversum{ x \circ p_y } f  = \oversum{x} \oversum{ y \circ p_x } f\circ switch $$
\end{dfn}
All the examples in the introduction, except the ordinals $Ord$ are commutative.

\begin{dfn}
	A dependent adder $A$ in $\mathscr{C}$ has a \textit{zero object} if there exists a morphism $0_A : * \rightarrow A$ such that
	\begin{enumerate}
		\item For all $x : U \rightarrow A$ let $const_{U,0_A}$ be the composite $U \rightarrow * \overset{0_A}{\rightarrow} A$ and let $const_{\llbracket x \rrbracket,0_A}$ be the composite $\llbracket x \rrbracket \rightarrow * \overset{0_A}{\rightarrow} A$. We demand that
		$$\oversum{x} const_{\llbracket x \rrbracket,0_A} = const_{U,0_A} $$
		\item For all $f : \llbracket 0_A \rrbracket \rightarrow A$
		$$\oversum{0_A} f = 0_A$$
	\end{enumerate}
\end{dfn}
All the examples in the introduction have a zero object.

\subsection{Categorical Structure} \label{sectionCategoricalStructure}

Let $A$ be a dependent adder in a category with pullbacks $\mathscr{C}$.

Given $x : U \rightarrow A$, $f : \llbracket x \rrbracket \rightarrow A$ and $g : \llbracket \oversum{x}f \rrbracket \rightarrow A$ we define the \textit{composition of $f$ and $g$} to be
$$ f \boxtimes g := \oversum{f} (g \circ f^\flat) $$
Then the Left Unit Axiom states $const_{1_A} \boxtimes f = f$.\\
The Sum Associativity Axiom states that $\oversum{\oversum{x}f }g = \oversum{x} f \boxtimes g $.\\
The Flatten Associativity Axiom states that the following diagram commutes:
	$$\xymatrix{
	\llbracket g \circ f^\flat \rrbracket \ar[d]_{ f^\flat \underset{A}{\times} id_F  } \ar[r]^{(g \circ f^\flat)^\flat} & \llbracket f \boxtimes g \rrbracket \ar[d]^{(f \boxtimes g)^\flat  } \\
	\llbracket g \rrbracket \ar[r]^(.45){g^\flat} & \llbracket \oversum{x} f \boxtimes g \rrbracket 
} $$

The $\boxtimes$ composition is associative in the following sense:
\begin{lem}\label{lemmaAssociativity}
	For every $x : U \rightarrow A$, $f : \llbracket x \rrbracket \rightarrow A$, $g : \llbracket \oversum{x} f \rrbracket \rightarrow A$ and $h : \llbracket \oversum{x} f \boxtimes g \rrbracket \rightarrow A$ we have that
	$$f \boxtimes (g \boxtimes h) = (f \boxtimes g) \boxtimes h  $$
\end{lem}
\begin{proof}
	By definition of $\boxtimes$ we have
	$$f \boxtimes (g \boxtimes h) = \oversum{f} (g \boxtimes h) \circ f^\flat = \oversum{f} (\oversum{g} h \circ g^\flat ) \circ f^\flat $$
	By naturality of dependent sums we have
	$$\oversum{f} (\oversum{g} h \circ g^\flat ) \circ f^\flat  = \oversum{f} \oversum{g \circ f^\flat} h \circ g^\flat \circ (f^\flat \underset{A}{\times}id_F) $$
	By the Flatten Associativity Axiom we have
	$$\oversum{f} \oversum{g \circ f^\flat} h \circ g^\flat \circ (f^\flat \underset{A}{\times}id_F) = \oversum{f} \oversum{g \circ f^\flat} h \circ (f \boxtimes g)^\flat \circ (g \circ f^\flat)^\flat $$
	By the Sum Associativity Axiom we have
	$$\oversum{f} \oversum{g \circ f^\flat} h \circ (f \boxtimes g)^\flat \circ (g \circ f^\flat)^\flat = \oversum{\oversum{f} g \circ f^\flat} h \circ (f \boxtimes g)^\flat = (f \boxtimes g) \boxtimes h$$
\end{proof}

We can associate to every dependent adder $A$ in $\mathscr{C}$ a associated category $Fib(A)$ internal to the presheaf category $\mathrm{PSh}(\mathscr{C})$. We call this the \textit{category of fibrations of $A$}, because in the case $A = Cat$ its global sections are equivalent to the category of small categories and Grothendieck fibrations.

A category in $\mathrm{PSh}(\mathscr{C})$ is the same thing as a presheaf of categories $Fib(A) : \mathscr{C}^{op} \rightarrow Cat$.

For $U \in \mathscr{C}$ define a category $Fib(A)(U)$ as follows:
The objects of $Fib(A)(U)$ are morphisms $x : U \rightarrow A$ in $\mathscr{C}$.
Given two objects $x : U \rightarrow A$ and $y : U \rightarrow A$ we define
$\mathrm{Hom}_{Fib(A)(U)}(x,y) := \{f \in \mathrm{Hom}_{\mathscr{C}}(\llbracket y \rrbracket, A)  | \overset{y}{\sum} f = x \}$.
We call elements of this set $A$-fibrations.
So an $A$-fibration $x \rightarrow y$ is a morphism $f : \llbracket y \rrbracket \rightarrow A$ whose sum is $x$.

For any object $x : U \rightarrow A$ we define $id_x$ to be the map $const_{1_A} : \llbracket x \rrbracket \rightarrow A$ that is the composite $\llbracket x \rrbracket \rightarrow * \overset{1_A}{\rightarrow} A$.

For two fibrations $g : x \rightarrow y$ and $f : y \rightarrow z$ we define their composition by:
$$f \circ g := f \boxtimes g = \oversum{f} (g \circ f^\flat) $$
This makes sense because $y = \oversum{z} f$, so $g$ is a function $\llbracket \oversum{z}f \rrbracket \rightarrow A$ and then the formula is well-defined.

The Sum Associativity Axiom of $A$ implies that the domain of $f \circ g$ is $x$.
The left and right unit axioms of the dependent adder $A$ imply the left and right unit axioms of the category $Fib(A)(U)$. The associativity of the composition of  $Fib(A)(U)$ follows from Lemma \ref{lemmaAssociativity}.

One can now easily check that a morphism $U \rightarrow V$ in $\mathscr{C}$ induces a functor $Fib(A)(V) \rightarrow Fib(A)(U)$, so that $Fib(A)$ becomes a presheaf of categories $Fib(A): \mathscr{C}^{op} \rightarrow Cat$. Equivalently this is also a category internal to $\mathrm{PSh}(\mathscr{C})$.

\subsection{Binary products}\label{sectionBinaryProducts}

Dependent sums imply binary products. For example in the ordinal numbers $Ord$ we have the formula
$$\alpha \cdot \beta = \depsum{i=1}{\beta} \alpha$$
and this same formula makes sense for arbitrary dependent adders.

Given a dependent adder $A$ in $\mathscr{C}$, and two parallel maps $x, y : U \rightarrow A$ in $\mathscr{C}$ we have a canonical map $p_y : \llbracket y \rrbracket \rightarrow U$ and define  $x \cdot y : U \rightarrow A$ by
$$x \cdot y := \oversum{y} x \circ p_y $$

With this we get a multiplication map
$$\mu_U : \mathrm{Hom}_{\mathscr{C}}(U,A) \times \mathrm{Hom}_{\mathscr{C}}(U,A) \rightarrow \mathrm{Hom}_{\mathscr{C}}(U,A) $$
which is natural in $U$.
By the Yoneda lemma we obtain a map $\mu : A \times A \rightarrow A$ in $\mathscr{C}$.
This map explicitly looks as follows:
let $\pi_2 : A \times A \rightarrow A$ be the second projection.
The first projection $\pi_2$ has a dependent sum function $\oversum{\pi_2} : \mathrm{Hom}_{\mathscr{C}}( \llbracket \pi_2 \rrbracket, A) \rightarrow \mathrm{Hom}_{\mathscr{C}}(A \times A, A)$.
Let $\pi_1 : \llbracket \pi_2 \rrbracket \cong F \times A \rightarrow A$ be the first projection.
The multiplication function $\mu : A \times A \rightarrow A$ is given by
$$\mu = \oversum{\pi_2} \pi_1$$

For all the examples from the introduction, this recovers the usual binary multiplication on the respective sets.

\begin{lem}
	$A$ is a monoid object in $\mathscr{C}$ with the multiplication function $\mu : A \times A \rightarrow A$ and the unit $1_A : * \rightarrow A$.\\
	If $A$ is a commutative dependent adder, then it is a commutative monoid object.
\end{lem}
\begin{proof}
	By the Yoneda lemma it suffices to show for every $U \in \mathscr{C}$ that $\mathrm{Hom}_{\mathscr{C}}(U,A)$ is a monoid, respectively a commutative monoid.
	
	The right unit axiom of the dependent adder $A$ implies the left unit axiom of the monoid $\mathrm{Hom}_{\mathscr{C}}(U,A)$, because for $x : U \rightarrow A$ we have
	$$const_{1_A} \cdot x = \oversum{x} const_{1_A} = x$$
	The Sum Associativity Axiom of the dependent adder implies the associativity of the monoid in the following way: For $x, y, z : U \rightarrow A$ we have a commutative diagram.
	$$\xymatrix{ \llbracket y \circ p_z \rrbracket\ar[d]_{(p_z \underset{A}{\times} id_F)}  \ar[rr]^{(y \circ p_z)^\flat} \ar[dr]^(.6){p_{y \circ p_z}} & & \llbracket y \cdot z \rrbracket \ar[d]^{p_{y \cdot z}} \\
	\llbracket y \rrbracket \ar@/_1.5pc/[rr]_{p_y} &	\llbracket z \rrbracket \ar[r]^{p_z} & U } $$
Here the upper right part commutes, because flattening maps commute over $U$. The lower part commutes because both maps are the projection $U \underset{A}{\times} F \underset{A}{\times} F \rightarrow U$.
	Using this diagram and the Sum Associativity Axiom we get
	$$x \cdot (y \cdot z) = \oversum{\oversum{z} y \circ p_z} x \circ p_{y \cdot z} = \oversum{z} \oversum{y \circ p_z} x \circ p_{y \cdot z} \circ (y \circ p_z)^\flat = \oversum{z} \oversum{y \circ p_z} x \circ p_y \circ (p_z \underset{A}{\times} id_F) = $$ $$ = \oversum{z} (\oversum{y} x \circ p_y) \circ p_z  = (x \cdot y) \cdot z$$
	
	For the right unit axiom of the monoid we need to work a bit harder, because we can only directly apply the left unit axiom of the dependent adder $A$ if $U$ is of the form $\llbracket w \rrbracket$ for some $w \in \mathscr{C}/A$. Thankfully we can use the naturality of the dependent sums and the unit $1_F : * \rightarrow F$ to translate the problem into such a situation.
	
	Let $c := const_1 : U \rightarrow A$ be the constant $1_A$ map.
	Consider the diagram
	$$\xymatrix{ U \ar@{..>}[dr]^s \ar[r] \ar@/_1.5pc/[rdd]_{id_U} & \textbf{*} \ar[dr]^{1_F} & \\
	             & \llbracket c \rrbracket \ar[r] \ar[d]^{p_c} & F \ar[d]_p \\
                 & U \ar[r]^{c} & A  }$$
    The lower right square is a pullback. The outer diragram commutes. So we get $s : U \rightarrow \llbracket c \rrbracket$ with $p_c \circ s = id_U$.
    Let $\tilde{c} := c \circ p_c$ be the constant $1_A$ function $\llbracket c \rrbracket \rightarrow A$.
    We have a commutative triangle
    $$\xymatrix{ U \ar[dr]_{\tilde{c} \circ s} \ar[rr]^s & & \llbracket c \rrbracket \ar[dl]^{\tilde{c}} \\
                 & A } $$
    so $s$ is a morphism $\tilde{c} \circ s \rightarrow \tilde{c}$ in $\mathscr{C}/A$.
    The naturality of dependent sums now implies that for every $z : \llbracket \tilde{c} \rrbracket \rightarrow A$ we have
    $$(\oversum{\tilde{c}} z) \circ s = \oversum{\tilde{c} \circ s} (z \circ (s \underset{A}{\times} id_F )) $$
    This in particular implies that for every $y : \llbracket c \rrbracket \rightarrow A$ we have $$(y \cdot \tilde{c}) \circ s=  (y \circ s) \cdot (\tilde{c} \circ s) $$
    
    Now take $x : U \rightarrow A$. We want to show that $x \cdot c = x$.
    Let $y := x \circ p_c$.
    We have $y \cdot \tilde{c} = \oversum{\tilde{c}} const_y = \oversum{\tilde{c}} const_y \circ \tilde{c}^\flat = y$ by the left unit axiom of the dependent adder $A$.
    Then $x = y \circ s = (y \cdot \tilde{c}) \circ s = (y \circ s) \cdot (\tilde{c} \circ s) = x \cdot c$.
    So $\mathrm{Hom}_{\mathscr{C}}(U,A)$ is a monoid.
	
	If the dependent adder is commutative, then the Fubini axiom and the left unit axiom of the monoid imply the commutativity of the monoid.
	$$x \cdot y = (1_A \cdot x) \cdot y = \depsum{}{y} \depsum{}{x \circ p_y} const_{1_A} =  \depsum{}{x} \depsum{}{y \circ p_x} const_{1_A} = y \cdot x$$
	
	With the Yoneda lemma it follows that $A$ is a monoid, respectively a commutative monoid, in $\mathscr{C}$.
\end{proof}

In every dependent adder $A$ we always have the right distributive law
$$y \cdot \depsum{i}{x} f(i) = \depsum{i}{x} y \cdot f(i) $$
but not necessarily the left distributive law. For example in $Ord$ we have
$(1 + 1) \cdot  \omega \neq (1 \cdot \omega) + (1 \cdot \omega)$.

In all examples of dependent adders in the introduction, the binary multiplication here recovers the usual binary multiplication on these sets. For example in $Cat$ the product $I \cdot J$ is the cartesian product of categories $I \times J$.

\section{Examples}\label{sectionExamples}

\subsection{Natural Numbers} \label{naturals}
Let $\mathscr{C} = Set$ be the category of small sets.
Let $A = \mathbb{N}$ be the set of natural numbers.
For $n \in \mathbb{N}$ let $[n] := \{1,\dots, n\}$ be a set with $n$ elements.
Let $F := \underset{n \in \mathbb{N}}{\coprod} [n] = \{(n,i) | n, i \in \mathbb{N}, 1 \leq i \leq n  \}$.
Let $p : F \rightarrow A$ be the map sending $(n,i)$ to $n$.
The fiber of $p$ over $n \in \mathbb{N}$ is the set $ [n]$. Let $1_F : * \rightarrow F$ be the point $(1,1) \in F$.
For any function $x : U \rightarrow \mathbb{N}$ can write $\llbracket x \rrbracket = U \underset{A}{\times} F$ as a coproduct:
$$\llbracket x \rrbracket = \underset{u \in U}{\coprod}[x(u)] $$

We define for every $x : U \rightarrow \mathbb{N}$ a natural function 
$$\oversum{x} : \mathrm{Hom}_{Set}( \llbracket x \rrbracket, \mathbb{N}) \rightarrow \mathrm{Hom}_{Set}(U,\mathbb{N}) $$
by sending $f : \llbracket x \rrbracket \rightarrow \mathbb{N}$ to the function $\oversum{x}f : U \rightarrow \mathbb{N}$ defined by
$$(\oversum{x} f) (u) := \depsum{i=1}{x(u)} f(u,i) $$

In the case $U = *$ this is the map $\mathrm{Hom}_{Set}( [x], \mathbb{N}) \rightarrow \mathbb{N}$ that sends $f$ to $\depsum{i=1}{x} f(i)$.

Next, for every $x : U \rightarrow \mathbb{N}$ and $f : \llbracket x \rrbracket \rightarrow \mathbb{N}$ we define the flattening map
$$f^\flat : \underset{u \in U}{\coprod} \underset{i \in [x(u)]}{\coprod} [f(u,i)] \rightarrow \underset{u \in U}{\coprod} [  \depsum{i=1}{x(u)} f(u,i)  ]  $$
by
$$f^\flat(u,i,j) := (u, j + \depsum{k=1}{i-1} f(u,k)) $$

The Right Unit axiom $\depsum{i=1}{x(u)}1=x(u)$ is obvious. 

The Left Unit axiom follows from the following calculation:
$$ (const_1 \boxtimes f)(u,i) = \depsum{j=1}{1} f(u, j + \depsum{k=1}{i-1} 1 ) = f(u,1 + i-1) = f(u,i)$$

The Sum Associativity Axiom
$$\depsum{j=1}{\depsum{i=1}{x(u)} f(u,i)} g(u,j) = \depsum{i=1}{x(u)} \depsum{j=1}{f(u,i)}  g(u, j + \depsum{k=1}{i-1} f(u,k) )  $$
is easily verified by induction over $n$.

For the Flatten Associativity Axiom we need to show that the following diagram commutes
$$\xymatrix{ \underset{u \in U}{\coprod} \underset{i \in [x(u)]}{\coprod} \underset{j \in [ f(u,i) ]  }{\coprod} [g(f^\flat(u,i,j)) ] \ar[r]^(.55){ (g \circ f^\flat)^\flat  } \ar[d]_{ f^\flat \underset{A}{\times} id_F } & \underset{u \in U}{\coprod}   \underset{i \in [x(u)]}{\coprod} [ \depsum{j=1}{f(u,i)} g(f^\flat(u,i,j))   ]  \ar[d]^{ (f \boxtimes g)^\flat } \\
   \underset{u \in U}{\coprod} \underset{ k \in [\depsum{i=1}{x(u)} f(u,i)]  }{\coprod} [g(u,k)]  \ar[r]^{g^\flat} &  \underset{u \in U}{\coprod} [ \depsum{i=1}{x(u)} \depsum{j=1}{f(u,i)} g(f^\flat(u,i,j))  ]  }   $$

The commutativity follows by the following calculation:
$$g^\flat((f^\flat \underset{A}{\times} id_F)  (u,i,j,k ) ) = g^\flat(f^\flat(u,i,j),k  ) = g^\flat(u, j + \depsum{a=1}{i-1} f(u,a),k ) = $$ 
$$ = (u,k + \depsum{b=1}{j-1 + \depsum{a=1}{i-1} f(u,a)} g(u,b)) = $$ $$= (u,k + (\depsum{b=1}{j-1} g(u,b + \depsum{a=1}{i-1} f(u,a))) + \depsum{a=1}{i-1} \depsum{b=1}{f(u,a)} g(u,b + \depsum{c=1}{a-1} f(u,c) ) ) = $$ $$ = (u, k + (\depsum{b=1}{j-1} g(f^\flat(u,i,b))) + \depsum{a=1}{i-1} (f \boxtimes g)(u,a) )= $$
$$= (f \boxtimes g)^\flat(u,i, k + \depsum{b=1}{j-1} g(f^\flat(i,b)) ) = (f \boxtimes g)^\flat((g \circ f^\flat)^\flat (u,i,j,k) ) $$
which means the Flatten Associativity Axiom is satisfied and $\mathbb{N}$ is a dependent adder.
The dependent adder is commutative because the Fubini axiom $\depsum{i=1}{x(u)} \depsum{j=1}{y(u)} f(u,i,j) = \depsum{j=1}{y(u)} \depsum{i=1}{x(u)} f(u,i,j) $ is obviously satisfied.

In the category of sets $\mathscr{C} = Set$ the entire $u$ variable is kind of unnecessary, and it is enough to define dependent sums, flattening maps and prove the axioms for $U = *$. The $u$ variable is however important in some other categories $\mathscr{C}$.

Let $FinSet$ be the full subcategory of $Set$ on sets of the form $ [n]$ for $n \in \mathbb{N}$.
The presheaves of categories $Fib(\mathbb{N}) : Set^{op} \rightarrow Cat$ is representable, and it is represented by the category $FinSet$.
By this we mean that for every $S \in Set$ we have a strict isomorphism of categories
$$Fib(\mathbb{N})(S) \cong FinSet^S $$
which is natural in $S$.

More precisely there is a bijection $$\Phi: \mathrm{Hom}_{FinSet}([n],[m]) \rightarrow \mathrm{Hom}_{Fib(\mathbb{N})(*)}(n,m) $$
sending $f : [n] \rightarrow [m]$ to the function $\Phi(f) : [m] \rightarrow \mathbb{N}$ that sends $x \in [m]$ to the cardinality of the fiber of $f$ over $x$.
$$\Phi(f)(x) := |f^{-1}(\{x\})| $$

It preserves composition in the sense that $\Phi(g \circ f) = \Phi(g) \boxtimes \Phi(f)$.

\subsection{Non-negative real numbers} \label{reals}

Let $\mathscr{C} = Top_{mtr}$ be the category of metrizable topological spaces and continuous maps between them. The category $\mathscr{C}$ has pullbacks which agree with the usual pullbacks of topological spaces.

Let $A := \bbR$ be the space of non-negative real numbers with euclidean topology.

Let $F := \{(x,y) \in \mathbb{R}^2 | 0 \leq x, 0 \leq y \leq x \}$ be equipped with the subspace topology from $\mathbb{R}^2$.
Define $p : F \rightarrow \bbR$, $p(x,y) := x$.
The fiber of $p$ over $x$ is the interval $[0,x]$.

We let $1_F : * \rightarrow F$ be the point $(1,1) \in F$.

For $x : U \rightarrow \bbR$ there is a homeomorphism
$$ \llbracket x \rrbracket \cong \{ (u, a) \in U \times \bbR | 0 \leq a \leq x(u) \} $$
and from now on we redefine $\llbracket x \rrbracket$ to mean the topological space $\{ (u, a) \in U \times \bbR | 0 \leq a \leq x(u) \}$.

For $x : U \rightarrow \bbR$ define
$\oversum{x} : \mathrm{Hom}_{Top_{mtr}}(\llbracket x \rrbracket, \bbR) \rightarrow \mathrm{Hom}_{Top_{mtr}}(U, \bbR) $
on $f: \llbracket x \rrbracket \rightarrow \bbR$ by
$$(\oversum{x} f)(u) := \underset{0}{\overset{x(u)}{\int}} f(u, t) dt$$
This integral always exists and is finite, because $f$ is continuous and for every $u \in U$ the interval $[0, x(u)]$ is compact. The map is easily seen to be natural in $x$, but we have to show that $\oversum{x}f$ is actually a continuous function. For this we will need to use the fact that $U$ is metrizable.

\begin{lem}\label{lemmaIntegralContinuous}
	The function $g:= \oversum{x} f : U \rightarrow \bbR$ is continuous.
\end{lem}
\begin{proof}
	Choose a metric $d$ on $U$, and put on $\llbracket x \rrbracket = U \underset{\bbR}{\times} F$ the metric that is the sum of the metric $d$ and the Manhattan distance on $F$.
	Take $u \in U$ and $\epsilon_1 > 0$.
	We want to show that the set $V:= g^{-1}(B_{\epsilon_1}(g(u)))$ is open in $U$.
	Since $U$ is metrizable, $U$ is also a compactly generated topological space. So to show that $V$ is open, it suffices to show for every compact subset $K \subseteq U$ that $K \cap V$ is open in $K$.
	So take an arbitrary compact subset $K \subseteq U$, and some $v \in K \cap V$.
	We need to show there exists $\delta > 0$ such that $K \cap B_\delta(v) \subseteq K \cap V$.
	Let $\epsilon_2 := \epsilon_1 - |g(v) - g(u)|$. We have $\epsilon_2 > 0$ and $B_{\epsilon_2}(g(v)) \subseteq B_{\epsilon_1}(g(u))$.
	Choose $\theta_1 > 0$ with $\theta_1 \cdot x(v) \leq \frac{\epsilon_2}{2}$.
	The set $K \underset{\bbR}{\times} F$ is a compact subset of $\llbracket x \rrbracket$.
	Therefore the restricted map $f : K \underset{\bbR}{\times} F \rightarrow \bbR$ is uniformly continuous. So there exists a $\delta_1 > 0$ such that for all $x, y \in K \underset{\bbR}{\times} F$, if $d(x,y) < \delta_1$ then $d(f(x),f(y)) < \theta_1$.
	Let $M$ be the maximum of $f : K \underset{\bbR}{\times} F \rightarrow \bbR$. Choose $\theta_2 > 0$ such that $\theta_2 \cdot M \leq \frac{\epsilon_2}{2}$ and $\theta_2 \leq \delta_1$.
	Since $x : U \rightarrow \bbR$ is continuous, there exists $\delta_2 > 0$ such that $x(B_{\delta_2}(v)) \subseteq B_{\theta_2}(x(v))$. Let $\delta$ be the minimum of $\delta_1$ and $\delta_2$.
	Take $w \in K \cap B_\delta(v)$. We claim that $w \in V$. This means that $g(w) \in B_{\epsilon_1}(g(u))$. To show that it suffices to show $g(w) \in B_{\epsilon_2}(v)$. We have
	$$g(w) = \underset{0}{\overset{x(w)}{\int}} f(w , t ) dt = \underset{x(v)}{\overset{x(w)}{\int}} f(w , t )  dt + \underset{0}{\overset{x(v)}{\int}} f(w  , t ) dt  $$
	Let's look at the first term of this sum:
	$$| \underset{x(v)}{\overset{x(w)}{\int}} f(w , t ) dt | \leq |x(w) - x(v) | \cdot M \leq \theta_2 \cdot M \leq \frac{\epsilon_2}{2}  $$
	With this we get
	$$|g(w) - g(v) | \leq \frac{\epsilon_2}{2} + | \underset{0}{\overset{x(v)}{\int}} f(w , t )    - f(v, t)  dt   | \leq $$
	$$\leq \frac{\epsilon_2}{2} + x(v) \cdot \theta_1 \leq \epsilon_2 $$
	so $g = \oversum{x} f $ is continuous.
\end{proof}

Given $x : U \rightarrow \bbR$ and $f : \llbracket x \rrbracket \rightarrow \bbR$ the flattening map is defined by
$$f^\flat : \underset{u \in U}{\coprod} \underset{a \in [0, x(u)]}{\coprod} [0, f(u,a)]   \rightarrow \underset{u \in U}{\coprod} [0, \underset{0}{\overset{x(u)}{\int}} f(u,t)dt  ] $$
$$f^\flat(u,a,b) := (u, \underset{0}{\overset{a}{\int}} f(u,t) dt  )$$
Here the $\coprod$ coproducts that we have written domain and codomain of $f^\flat$ are not literally coproducts of topological spaces, but they are coproducts of sets equipped with the subspace topology from $U \times \bbR \times \bbR$, respectively $U \times \bbR$.
Also note that $f^\flat(u,a,b)$ does not depend on $b$ at all, so if we have continuous functions $x : U \rightarrow \bbR$, $f: \llbracket x \rrbracket \rightarrow \bbR$ and $g : \llbracket \oversum{x}f \rrbracket \rightarrow \bbR$, then $$(f \boxtimes g)(u,a) = \underset{0}{\overset{f(u,a)}{\int}} g(f^\flat(u,a,t))dt = f(u,a) \cdot g( \underset{0}{\overset{a}{\int}} f(u,s)ds) $$

We now verify that $\bbR$ satisfies the axioms of a commutative dependent adder.

The Right Unit axiom follows from the fact that $$\underset{0}{\overset{x}{\int}} 1 dt = x$$
The Left Unit axiom follows from 
$$ (const_{1_A} \boxtimes f)(u,a) = 1 \cdot f(u, \underset{0}{\overset{a}{\int} 1} ) = f(u,a)$$

For the Sum Associativity Axiom we need to show that for every continuous function $f: \llbracket x \rrbracket  \rightarrow \bbR$ and $g : \llbracket \oversum{x}f \rrbracket \rightarrow \bbR$ we have
$$ \underset{0}{\overset{  \underset{0}{\overset{x}{\int}} f(u,t)dt  }{\int}} g(u,s) ds =  \underset{0}{\overset{x}{\int}}  f(u,t) \cdot g(u,  \underset{0}{\overset{t}{\int}} f(u,z)dz )  dt  $$

This follows through \enquote{integration by substitution} with an antiderivative of $f$.
Let $h:\llbracket x \rrbracket \rightarrow \bbR$, $h(u,t) := \underset{0}{\overset{t}{\int}} f(u,s)ds$. Then $h$ is continuously differentiable in the $t$ variable with $\frac{d}{dt}h(u,t) = f(u,t)$, and we get
$$\underset{0}{\overset{  \underset{0}{\overset{x}{\int}} f(u,t)dt  }{\int}} g(u,s) ds = \underset{h(u,0)}{\overset{  h(u,x)  }{\int}} g(u,s) ds = \underset{0}{\overset{x}{\int}}  f(u,t) \cdot g(u,h(u,t))  dt $$

The Flatten Associativity Axiom follows from
$$g^\flat((f^\flat \underset{A}{\times} id_F) (u,a,b,c)) = g^\flat(u,\underset{0}{\overset{a}{\int}} f(u,t)dt , c  ) = \underset{0}{\overset{\underset{0}{\overset{a}{\int}} f(u,t)dt}{\int}} g(u,t)dt = $$
$$ = \underset{0}{\overset{a}{\int}}  (f \boxtimes g)(u,t) dt = (f \boxtimes g)^\flat(u, a, \underset{0}{\overset{b}{\int}} g(f^\flat(i,t)) dt) = (f \boxtimes g)^\flat((g \circ f^\flat)^\flat(u,a,b,c)  )  $$
so $\bbR$ is a dependent adder.

It is a commutative dependent adder because of Fubini's theorem $$\underset{0}{\overset{x}{\int}} \underset{0}{\overset{y}{\int}} f(u,s,t) ds dt = \underset{0}{\overset{y}{\int}} \underset{0}{\overset{x}{\int}} f(u,s,t) dt ds$$

\subsection{P-Adic integers}

Like in Section \ref{reals}, we take $\mathscr{C} = Top_{mtr}$ the category of metrizable spaces.

Let $q$ be a prime number.
Let $A := \mathbb{Z}_q$.
Let $F := \mathbb{Z}_q \times \mathbb{Z}_q$.
Let $p : F \rightarrow A$ be the second projection $\mathbb{Z}_q \times \mathbb{Z}_q \rightarrow \mathbb{Z}_q$.
The fiber of $p$ over $x \in \mathbb{Z}_q$ is $\mathbb{Z}_q$.
Let $1_F: * \rightarrow F$ be the point $(1,1) \in \mathbb{Z}_q \times \mathbb{Z}_q$.
For $x : U \rightarrow \mathbb{Z}_q$ there is a homeomorphism $\llbracket x \rrbracket \cong U \times \mathbb{Z}_q$, so by $\llbracket x \rrbracket$ we will from now on simply mean $U \times \mathbb{Z}_q$.

To define the dependent sums we use the following lemma
\begin{lem}\label{lemmaZpsums}
	If $f : \mathbb{Z}_q \rightarrow \mathbb{Z}_q$ is a continuous function,
	and $g : \mathbb{N} \rightarrow \mathbb{Z}_q$ is the function
	$g(n) := \depsum{i=1}{n} f(i)$
	then $g$ is continuous with respect to the $q$-adic topology on $\mathbb{N}$.
\end{lem}
\begin{proof}
	Take $N \in \mathbb{N}$. We need to show that there exists $M \in \mathbb{N}$ such that for all $x, y \in \mathbb{N}$ with $x-y$ divisible by $q^M$ we have that $g(x) - g(y)$ is divisble by $q^N$.
	
	Since $\mathbb{Z}_q$ is compact, $f$ is uniformly continuous. So there exists a $K \in \mathbb{N}$, such that for all $x, y \in \mathbb{Z}_q$ if $x- y \in q^K\mathbb{Z}_q$ then $f(x)-f(y) \in q^N\mathbb{Z}_q$.
	Let $M := K + N$.
	Take $x, y \in \mathbb{N}$ with $x - y = q^M\cdot z$ for some $z \in \mathbb{N}$. Without loss of generality assume $x \geq y$.
	Then we have
	$$g(x) - g(y) = \depsum{i=y+1}{x} f(i) = \depsum{i=1}{x-y} f(i+y) = \depsum{i=1}{q^M z} f(i+y) = \depsum{i=1}{q^N} \depsum{j=1}{q^K z} f(j + y + (i-1)q^K z) $$
	
	Now for all $i,j$ there exists $w_{i,j}$ such that $f(j + y + (i-1)q^K z) = f(j + y) + q^Nw_{i,j}$, and then
	$$\depsum{i=1}{q^N} \depsum{j=1}{q^K z} f(j + y + (i-1)q^K z) = q^N( \depsum{j=1}{q^K z} f(j + y) ) + q^N( \depsum{i=1}{q^N} \depsum{j=1}{q^K z} w_{i,j})$$ 
	This is divisible by $q^N$ so $g$ is continuous.
\end{proof}

In the above lemma, since $\mathbb{Z}_q$ is Cauchy-complete, the map $g : \mathbb{N} \rightarrow \mathbb{Z}_q$ induces a map from the Cauchy completion of $\mathbb{N}$ to $\mathbb{Z}_q$. Since the Cauchy completion of $\mathbb{N}$ with respect to the $q$-adic topology is $\mathbb{Z}_q$, we thus get a continuous map $\tilde{g} : \mathbb{Z}_q \rightarrow \mathbb{Z}_q$, $\tilde{g}(x) = \underset{\underset{x_n \in \mathbb{N}}{x_n \rightarrow x}}{\lim} g(x)$.
We write $\depsum{i=1}{x} f(i)$ for $\tilde{g}(x)$.
$$ \depsum{i=1}{x} f(i) :=  \underset{\underset{x_n \in \mathbb{N}}{x_n \rightarrow x}}{\lim} \depsum{i=1}{x_n} f(i)$$

We can now define the dependent sum function of $\mathbb{Z}_q$.
For any continuous function $x : U \rightarrow \mathbb{Z}_q$ define
$$\oversum{x} : \mathrm{Hom}_{Top_{mtr}}(U \times \mathbb{Z}_q, \mathbb{Z}_q) \rightarrow \mathrm{Hom}_{Top_{mtr}}(U , \mathbb{Z}_q)  $$
by $$(\oversum{x}f)(u) := \depsum{i=1}{x(u)} f(u,i) $$

The continuity of this function can be proven using an argument very similar to the one from Lemma \ref{lemmaIntegralContinuous}.

For any $f : U \times \mathbb{Z}_q \rightarrow \mathbb{Z}_q$ we define the flattening function of $f$ by
$$f^\flat : U \times \mathbb{Z}_q \times \mathbb{Z}_q \rightarrow \mathbb{Z}_q, f^\flat(u,i,j) := j + \depsum{k=1}{i-1} f(u,k) $$

The right and left unit axiom, both Associativity Axioms and Fubini axiom follow exactly like for the natural numbers $\mathbb{N}$ from Section \ref{naturals}.

\subsection{Small Categories}\label{categories}

We first explicitly construct oplax colimits of categories.
We then show that $Cat$ is a dependent adder with oplax colimits as dependent sums.
We then show that $Cat^{op}$ is a dependent adder with lax colimits / Grothendieck constructions as dependent sums.

\subsubsection{Explicit construction of oplax colimits}

For any functor $f : I \rightarrow Cat$ define its \textit{oplax colimit} $\oplaxcolim{i \in I} f(i)$  to be the following category:
\begin{enumerate}
	\item 
	The objects are pairs $(i,j)$ with $i \in I$, $j \in f(i)$.
	\item 
	The morphism $(i,j) \rightarrow (i^\prime, j^\prime)$ are pairs of morphisms $(\alpha,\beta)$ where $\alpha : i \rightarrow i^\prime$ in $I$ and $\beta : f(\alpha)(j) \rightarrow j^\prime$ in $f(i^\prime)$.
\end{enumerate}

The universal property of this particular oplax colimit construction is as follows:
For any $i \in I$ we have a functor $\iota_i : f(i) \rightarrow \oplaxcolim{j \in I} f(j)$. For every $\alpha : i \rightarrow i^\prime$ in $I$ we have a natural transformation $\iota_{\alpha} : \iota_i \rightarrow \iota_{i^\prime} \circ f(\alpha)$, such that $\iota_{id_i} = id_{\iota_i}$ and $\iota_{\beta \circ \alpha} = (\iota_{\beta}*f(\alpha)) \circ \iota_{\alpha}$.

Whenever we have another object $X$ together with for every $i \in I$ a functor $\omega_i : f(i) \rightarrow X$, and for every $\alpha : i \rightarrow i^\prime$ in $I$ a natural transformation $\omega_\alpha: \omega_i \rightarrow \omega_{i^\prime} \circ f(\alpha) $ such that $\omega_{id_i} = id_{\omega_i}$ and $\omega_{\beta \circ \alpha } = (\omega_\beta * f(\alpha)) \circ \omega_\alpha$ then there is a unique functor $\omega : \oplaxcolim{j \in I} f(j) \rightarrow X$, such that $\omega_i = \omega \circ \iota_i$ and $\omega_\alpha = \omega * \iota_\alpha$.

Usually the universal property of oplax colimits requires the functor $\omega$ to not be unique, but only be unique up to isomorphism. However the particular category we constructed above satisfies the stricter universal property where $\omega$ is unique up to equality.
See \cite[Section 2]{bird1989flexible} for the usual theory of lax and oplax limits.

\subsubsection{Oplax colimits as sums}

We postulate two strongly inaccessible cardinals $\kappa_0 < \kappa_1$.
A category is called \textit{small} if it is smaller than $\kappa_0$.
A category is called \textit{moderately small} if it is smaller than $\kappa_1$.

The $1$-category of small categories is denoted $Cat$.
The $1$-category of moderately small categories is denoted $CAT$.

For all moderately small categories $I, J$ let $\mathrm{Hom}_{hCAT}(I,J)$ be a skeleton of the category of functors $\mathrm{Hom}_{CAT}(I,J)$. So for every functor $F : I \rightarrow J$ there exists exactly one functor $\mathrm{Ho}(F) : I \rightarrow J$ such that $F \cong \mathrm{Ho}(F)$ and $\mathrm{Ho}(F) \in \mathrm{Hom}_{hCAT}(I,J)$.

Let $hCAT$ be the category whose objects are moderately small categories, and whose morphism sets are given by $\mathrm{Hom}_{hCAT}(I,J)$.
Given two functors $\mathrm{Ho}(F) : I \rightarrow J$ and $\mathrm{Ho}(G) : J \rightarrow K$ their composition is defined by taking $sk$ of their composition in $CAT$:
$$\mathrm{Ho}(F) \underset{hCAT}{\circ} \mathrm{Ho}(G) := \mathrm{Ho}( \mathrm{Ho}(F) \underset{CAT}{\circ} \mathrm{Ho}(G)  )$$
With this composition $hCAT$ is a category.

Then $hCAT$ is in fact equivalent to the homotopy category of $CAT$, if we put on $CAT$ the canonical model structure constructed in \cite{rezk1996model}, in which weak equivalences are categorical equivalences and fibrations are Joyal isofibrations.
We just prefer to present the morphisms of $hCAT$ as particular chosen functors instead of presenting them as isomorphism classes of functors.

The category $hCAT$ has finite limits. In general pullbacks in $hCAT$ do not need to coincide with pullbacks in $CAT$. However if we take a pullback of a Grothendieck fibration or opfibration in $CAT$, then it will also be a pullback in $hCAT$, because Grothendieck fibrations and opfibrations are Joyal isofibrations, and in any right proper model category a strict pullback of a fibration is a homotopy pullback.

Let $\mathscr{C} := hCAT$ be the homotopy category of moderately small categories.
Let $A:= Cat$ be the $1$-category of small categories.

Consider the inclusion functor $\iota: A \rightarrow CAT$, $\iota(I) := I$. This functor corresponds to a Grothendieck op-fibration $p^{\heartsuit} : F \rightarrow A$. We define $p := \mathrm{Ho}(p^{\heartsuit})$. So we now have a morphism $p : F \rightarrow A$ in $\mathscr{C}$.
The fiber of $p$ over some $I \in A$ is isomorphic to $I \in CAT$. Define $1_F^{\heartsuit} : 1 \rightarrow F$ as the map that picks out the terminal category $1 \in A$ and the unique object in that category $* \in 1$. Define the unit $1_F$ by $1_F := \mathrm{Ho}(1_F^{\heartsuit})$.

Since pullbacks of Grothendieck op-fibrations are Grothendieck op-fibrations, we have for any functor $x : U \rightarrow A$ in $hCAT$ an isomorphism
$$\llbracket x \rrbracket \cong \oplaxcolim{u \in U} x(u)  $$

For every $x : U \rightarrow A$ we define the dependent sum function
$$\oversum{x} : \mathrm{Hom}_{\mathscr{C}}(\llbracket x \rrbracket, A) \rightarrow \mathrm{Hom}_{\mathscr{C}}(U, A) $$
on a functor $f: \llbracket x \rrbracket\rightarrow A$ as follows: 

We define a functor $\oversum{x}(f)^\heartsuit: U \rightarrow A$ that is defined on objects by  $$\oversum{x}(f)^\heartsuit(u) := \oplaxcolim{i \in x(u)} f(u,i)$$
Given an arrow $\alpha: u \rightarrow v$ in $U$ we need to
define an arrow $\oversum{x}(f)^\heartsuit(\alpha): \oversum{x}(f)^\heartsuit(u) \rightarrow \oversum{x}(f)^\heartsuit(v)$ in $A$, by using the universal property of the oplax colimit in the codomain.
For all $i \in x(u)$ we then have an arrow $(\alpha, id_{x(\alpha)(i)} ) : (u, i) \rightarrow (v, x(\alpha)(i))$ in $\oplaxcolim{u \in U} x(u) \cong \llbracket x \rrbracket$. We then get a composite functor $ f(u,i) \overset{f(\alpha, id_{x(\alpha)(i)})}{\rightarrow} f(v, x(\alpha)(i)) \rightarrow \oplaxcolim{k \in x(v)} f(v,k)  $.
For every morphism $\beta : i \rightarrow j$ in $x(u)$ we have a diagram
$$\xymatrix{  f(u,i) \ar[rr]^(.45){f(\alpha, id_{x(\alpha)(i)} )}  \ar[d]_{f(id_u,\beta)} & & f(v, x(\alpha)(i))  \ar[d]_{f(id_v,x(\alpha)(\beta)  )}  \ar[dr]_(.45){}="1" \\
	          f(u,j)  \ar[rr]_(.45){f(\alpha, id_{x(\alpha)(j)})} & & f(v, x(\alpha)(j)) \ar[r]^(.45){}="2" & \oplaxcolim{k \in x(v)} f(v,k)  \ar@{=>}|-{}"1" ;"2"   }$$
where the left square commutes, and in the right triangle there is a canonical natural transformation.

By the universal property of the oplax colimit $\oversum{x}(f)^\heartsuit(u) = \oplaxcolim{i \in x(u)} f(u,i)$ we get a morphism $ \oversum{x}(f)^\heartsuit(\alpha) :  \oversum{x}(f)^\heartsuit(u) \rightarrow \oversum{x}(f)^\heartsuit(v)$.
Then $\oversum{x}(f)^\heartsuit : U \rightarrow A$ is a functor.
We define $\oversum{x}(f) := \mathrm{Ho}(\oversum{x}(f)^\heartsuit)$.

Next we need to define the flattening map
$$f^\flat : \llbracket f \rrbracket \rightarrow \llbracket \oversum{x}(f) \rrbracket$$

We have canonical isomorphisms
$$\llbracket f \rrbracket \cong \oplaxcolim{(u,i) \in \llbracket x \rrbracket} f(u,i) \cong \oplaxcolim{(u,i) \in \oplaxcolim{u \in U} x(u)} f(u,i) $$
$$ \llbracket \oversum{x} f \rrbracket \cong \oplaxcolim{u \in U}  \oplaxcolim{i \in x(u)} f(u,i)  $$
and using the universal property of oplax colimits one can construct a natural isomorphism of categories
$$\oplaxcolim{(u,i) \in \oplaxcolim{u \in U} x(u)} f(u,i) \overset{\sim}{\rightarrow} \oplaxcolim{u \in U}  \oplaxcolim{i \in x(u)} f(u,i) $$

So there is a canonical isomorphism
$f^{\flat,\heartsuit} : \llbracket f \rrbracket \rightarrow \llbracket \oversum{x} f \rrbracket $
and we define $f^\flat := \mathrm{Ho}( f^{\flat,\heartsuit} )$.

We now verify that $A$ satisfies all the axioms of a dependent adder.

For any category $I$ we have a natural isomorphism
$$\oplaxcolim{i \in I} 1 \cong I$$
and this implies that for any map $x : U \rightarrow Cat$ in $hCAT$ we have a strict equality $\oversum{x} const_{1_A} = x$ of morphisms in $hCAT$, so it satisfies the Right Unit Axiom.

The Left Unit Axiom follows from the fact that if $I$ is a small category and $F_I : 1 \rightarrow Cat$ is the functor sending the unique object $*$ from the terminal category $1$ to $I \in Cat$, then $I$ is the oplax colimit of $F_I$.

$$\oplaxcolim{i \in 1} F_I(i) \cong F_I(*) = I $$

The Sum Associativity Axiom follows from the fact that 
for any category $I$, functor $F: I \rightarrow Cat$ and $G : \oplaxcolim{i \in I} F(i) \rightarrow Cat$ there is a natural isomoprhism
$$\oplaxcolim{(i,j) \in \oplaxcolim{i \in I} F(i)  } G(i,j) \rightarrow \oplaxcolim{i \in I} \oplaxcolim{j \in F(i)} G(i,j)$$

For the Flatten Associativity Axiom one needs to verify
for every $x : U \rightarrow Cat$, $f : \oplaxcolim{u \in U} x(u) \rightarrow Cat$ and $g : \oplaxcolim{u \in U} \oplaxcolim{i \in x(u)} f(u,i) \rightarrow Cat$
that the following diagram commutes up to natural isomorphism
\footnotesize
$$\xymatrix{ \oplaxcolim{(u,i,j) \in \oplaxcolim{(u,i) \in \oplaxcolim{u \in U} x(u)  } f(u,i)   }  g(u,i,j)  \ar[d]  \ar[r]  &  \oplaxcolim{(u,i) \in \oplaxcolim{u \in U} x(u)  }  \oplaxcolim{j \in f(u,i)} g(u,i,j) \ar[d]    \\
	 \oplaxcolim{ (u,i,j) \in \oplaxcolim{ u \in U  } \oplaxcolim{i \in x(u)} f(u,i)     } g(u,i,j)  \ar[r] &   \oplaxcolim{u \in U}  \oplaxcolim{i \in x(u)} \oplaxcolim{j \in f(u,i) } g(u,i,j)   } $$
\normalsize

where all maps are constructed canonically using the universal properties of the oplax colimits.
If this diagram commutes up to isomorphism in $CAT$ then the corresponding diagram in $hCAT$ commutes strictly, and then $Cat$ satisfies the Flatten Associativity Axiom.

$Cat$ also satisfies the Fubini axiom because oplax colimits commute, so it is a commutative dependent adder.

\subsubsection{Lax colimits as sums}
One can also make the category of small categories $Cat$ into a dependent adder with lax colimits instead of oplax colimits.

For a functor $F : I \rightarrow CAT$ we can define the lax limit $\laxcolim{i \in I} F(i)$ of $F$ as the opposite of the oplax colimit of the composite $I \overset{F}{\rightarrow} CAT \overset{op}{\rightarrow} CAT$.
$$ \laxcolim{i \in I} F(i) := (\oplaxcolim{i \in I} F(i)^{op} )^{op}$$
This is also known as the Grothendieck construction $\int F$.

We have a functor $\iota : Cat \rightarrow CAT$ defined by $\iota(I) := I^{op}$. It corresponds to a Grothendieck op-fibration $p : F \rightarrow Cat$, whose fiber over some $I \in Cat$ is isomorphic to $I^{op}$.

Given a functor $F: I^{op} \rightarrow Cat$ we define the sum of $F$ to be its lax colimit / Grothendieck construction.
$$\oversum{I} F := \laxcolim{i \in I^{op}} F(i) = \int F $$

We then have natural isomorphisms
$$\laxcolim{i \in I^{op}} 1 \cong I $$
$$ \laxcolim{ (i,j) \in (\laxcolim{i \in I^{op}} F(i))^{op}  } G(i,j) \cong \laxcolim{i \in I^{op}}  \laxcolim{j \in F(i)^{op}} G(i,j)   $$
which make $Cat$ into a dependent adder in $hCAT$.

\subsection{Commutative $\mathbb{Q}$-algebras}

Let $R$ be a commutative $\mathbb{Q}$-algebra. Let $\mathscr{C} := Sch/R$ be the category of schemes over $\mathrm{Spec}(R)$. Let $A := \mathbb{A}^1_R$ be the affine line.
We define $F:= \mathbb{A}^2_R$ and let $p : F \rightarrow \mathbb{A}^1_R$ be the projection to the second variable.
For every $x : U \rightarrow \mathbb{A}^1_R$ we have $\llbracket x \rrbracket \cong \mathbb{A}^1_U$. Let $S := \Gamma(U, \mathscr{O}_U)$ be the ring of global sections of $U$.
Then morphisms $x : U \rightarrow \mathbb{A}^1_R$ correspond to elements $x \in S$, and morphisms $p : \mathbb{A}^1_U \rightarrow \mathbb{A}^1_R$ corresponds to elements $p \in S[X]$.
So for every $x \in S$ and every $p \in S[X]$ we need to define a sum $\oversum{x}p \in S$. For this we will use Faulhaber's formula \cite{orosi2018simple}.
For any $d \in \mathbb{N}$ define the $d$-th Faulhaber polynomial $F_d \in S[X]$ by
$$F_d := \frac{1}{d+1} \depsum{n=0}{d}\binom{d+1}{n}B_nX^{d-n+1} $$
where $\binom{d+1}{n}$ is the binomial coefficient and $B_n$ is the Bernoulli number with the convention $B_1= +\frac{1}{2}$.
Faulhaber's formula states that for any $d, n \in \mathbb{N}$
$$\depsum{k=1}{n} k^d = F_d(n)$$

For any polynomial $p \in S[X]$ we can write $p$ as a sum of monomials $p = \depsum{i=0}{n} p_i\cdot X^i$ and then define $\sum(p) \in S[X]$ by
$$\sum (p) := \depsum{i=0}{n} p_i \cdot F_i$$
Then $\sum(p)$ is a polynomial satisfying $\sum(p)(m) = \depsum{i=1}{m} p(i)$ for all $m \in \mathbb{N}$.

We define for all $x \in S$ and $p \in S[X]$ that $\oversum{x}(p) := \sum(p)(x)$.

With this we have defined the dependent sum function of the dependent adder $\mathbb{A}^1_R$.

Next, for every $p \in S[X]$ we need to define a flattening function $p^\flat : \mathbb{A}^2_U \rightarrow \mathbb{A}^1_U$ over $U$. Such a morphism corresponds to an element in $S[X,Y]$.
We define $p^\flat \in S[X,Y]$ by $p^\flat := Y + \oversum{X-1}(p)$.

We now need to verify the axioms of a dependent adder.
The Left Unit Axiom is true because $\oversum{x}1 = F_0(x) = x$.
The Right Unit Axiom is true because $\oversum{const_1}(p\cdot const_1^\flat) = p(1 + \oversum{X-1} 1) = p(X) = p$.

For the Sum Associativity Axiom we need to show for all $R$-algebras $S$ and all $x \in S$, $f, g \in S[X]$, that
$\oversum{\oversum{x}(f)}(g) = \oversum{x}(\oversum{f} g \circ f^\flat)$ in $S$.
To show it we will use the following lemma:
\begin{lem}\label{lemmaPolyEqual}
	Let $B$ be a commutative $\mathbb{Q}$-algebra.
	If for two polynomials $p, q \in B[X]$ we have $p(n) = q(n)$ for all $n \in \mathbb{N}$, then $p = q$ in $B[X]$.
\end{lem}
\begin{proof}
	We claim that for all $f \in B[X]$, if $f(n) = 0$ for all $n \in \mathbb{N}$, then $f = 0$ in $B[X]$. Once we have shown this claim the lemma follows with $f := p-q$.
	We show this claim by induction over the degree $d$ of $f$.
	
	Suppose that for all polynomials $g$ of degree $d$, if $g(n) = 0$ for all $n \in \mathbb{N}$, then $g = 0$.
	
	Let $f \in B[X]$ be a polynomial of degree $d+1$ such that $f(n) = 0$ for all $n \in \mathbb{N}$.
	Write $f = \depsum{i=0}{d+1} f_iX^i$.
	Then $0 = f(0) = f_0$. So $f = X \cdot (\depsum{i=1}{d+1} f_iX^{i-1})$.
	Let $g := \depsum{i=1}{d+1} f_iX^{i-1}$. For all $n \in \mathbb{N}$ with $n \neq 0$ we have $g(n) = \frac{f(n)}{n} = 0$. So $g(X+1) = \depsum{i=1}{d+1} f_i(X+1)^{i-1}$ is a polynomial of degree $d$ that vanishes on all of $\mathbb{N}$. By inductive hypothesis $g(X+1) = 0$ in $B[X]$. This then implies $g = 0$ and then $f = 0$.
\end{proof}

For $p \in S[T]$ and $n \in \mathbb{N}$ we know that $\oversum{n+1}(p) = p(n+1) + \oversum{n}(p)$.
With Lemma \ref{lemmaPolyEqual} it follows that for all $x \in S$ we have $\oversum{x+1}(p) =  p(x+1)+\oversum{x}(p) $.
We can then show inductively for all $n \in \mathbb{N}$, $x \in S$ and $p \in S[T]$ that $\oversum{x+n}(p) = \oversum{x}(p) + \oversum{n}(p( T + x) )$.
By Lemma \ref{lemmaPolyEqual} it follows that for all $x, y \in S$ and $p \in S[T]$ that $\oversum{x+y}(p) = \oversum{x}(p) + \oversum{y}(p(T + x))$.
We can then inductively show for all $n \in \mathbb{N}$, and $f, g \in S[T]$ that
$$\oversum{\oversum{n}(f)}(g) = \depsum{j=1}{\depsum{i=1}{n}f(i)}g(j) = \depsum{i=1}{n} \depsum{j=1}{f(i)} g(j+\depsum{k=1}{i-1}f(k)) = \oversum{n}(\oversum{f}(g \circ f^\flat))$$
From Lemma \ref{lemmaPolyEqual} we then get the Sum Associativity Axiom for all $x \in S$ and all $f, g \in S[T]$.

As soon as one has the Sum Associativity Axiom, one can prove the Flatten Associativity Axiom exactly like we did for the natural numbers in Section \ref{naturals}.

We also claim that $\mathbb{A}^1_R$ satisfies the Fubini axiom. This can be shown as follows.
Firstly we have for all $n \in \mathbb{N}$ and $p, q \in S[T]$ that $\oversum{n}(p+q) = \oversum{n}(p) + \oversum{n}(q)$. So Lemma \ref{lemmaPolyEqual} implies for all $p , q \in S[T]$ and $x \in S$ that $\oversum{x}(p+q) = \oversum{x}(p) + \oversum{x}(q)$.
One can then show inductively that for all $n \in \mathbb{N}$, $x \in S$ and $p \in S[T_1][T_2]$ that $\oversum{x}(\oversum{n} (p) ) = \oversum{n}(\oversum{x}(switch(p) )$ where $switch: S[T_1][T_2] \rightarrow S[T_1][T_2]$ is isomorphism exchanging $T_1$ and $T_2$, and then the Fubini axiom follows from Lemma \ref{lemmaPolyEqual}.

\subsection{All the other examples}
\begin{enumerate}
	\item The set of small ordinals $Ord$ is a dependent adder in the category of moderately small sets $SET$.
	We take the function $p : F \rightarrow Ord$ whose fiber $p^{-1}(\{\alpha\})$ at some $\alpha \in Ord$ is the underlying set of $\alpha$.
	For every $\alpha$ we define the $\alpha$-dependent sum function $$\oversum{\alpha}: \mathrm{Hom}_{SET}(\alpha, Ord) \rightarrow Ord$$
	by transfinite induction over $\alpha$.
	$\oversum{0}$ is defined by $\oversum{0} f := 0$.
	$\oversum{\alpha+1}$ is defined by $\oversum{\alpha+1}f := (\oversum{\alpha} (f|_\alpha) ) + f(\alpha)$. Here $f|_\alpha$ means the restriction of $f$ to $\alpha \subseteq \alpha +1$. Also, it is important here to put the $f(\alpha)$ on the right of the sum, and not on the left.
	If $\alpha = \underset{\beta < \alpha }{\cup} \beta$ is a limit ordinal, then we define $\oversum{\alpha} f := \underset{\beta < \alpha }{\mathrm{sup}} \oversum{\beta} f$.
	These $\alpha$-dependent sum functions then assemble together for every $x : U \rightarrow Ord$ into an $x$-dependent sum function
	$$\oversum{x} : \mathrm{Hom}_{SET}(\llbracket x \rrbracket, Ord) \rightarrow \mathrm{Hom}_{SET}(U, Ord) $$
	
	Note that the indexing conventions here are slightly shifted compared to those from Section \ref{naturals}, because here $\oversum{\alpha}f$ refers to the sum of $f(i)$ for all $0 \leq i < \alpha$, while in Section \ref{naturals} the sum $\oversum{n} f$ referred to the sum of $f(i)$ for all $1 \leq i \leq n$.
	
	For $f : \llbracket x \rrbracket \rightarrow Ord$ the flattening function $f^\flat : \underset{u \in U}{\coprod} \underset{\alpha \in x(u)}{\coprod} f(\alpha) \rightarrow \underset{u \in U}{\coprod} \oversum{x(u)}(f) $ is defined by $f^\flat(u,i,j) := (u,  (\oversum{i} (f)) + j)$. It is important that $j$ is here on the right side of the sum.
	
	One can then easily verify all the axioms of dependent adders by using a lot of transfinite induction.
	
	\item The \enquote{set of all small cardinals} $Card$ is a dependent adder in the category of moderately small sets $SET$. It works exactly like the ordinal numbers $Ord$.
	\item The \enquote{field with one element} $\mathbb{F}_1 = \{0,1\}$ works exactly like the natural numbers $\mathbb{N}$. See Section \ref{naturals}.
	\item The unit interval $[0,1]$ works exactly like the real numbers $\bbR$. See Section \ref{reals}.
	\item The integers $\mathbb{Z}$ are a dependent adder in $Set$.
		We define $p : F \rightarrow \mathbb{Z}$ to be the function whose fibers are $\mathbb{Z}$ everywhere. So $p$ is the second product projection $p : \mathbb{Z} \times \mathbb{Z} \rightarrow \mathbb{Z}$.
		Given some $x : U \rightarrow \mathbb{Z}$ and a function $f : U \times \mathbb{Z} \rightarrow \mathbb{Z}$ we define its sum by
		$$\oversum{x}(f)(u) := \begin{cases}
			\depsum{i=1}{x(u)} f(u,i) &, x(u) \geq 0\\
			\depsum{i=1}{-x(u)}-f(u,1-i) &, x(u) < 0
		\end{cases}$$
	The flattening function $f^\flat : U \times \mathbb{Z} \times \mathbb{Z} \rightarrow U \times \mathbb{Z}$ is given by $f^\flat(u, i, j) := (u, j + \depsum{k=1}{i-1}f(u,k))$.
	This makes $\mathbb{Z}$ into a dependent adder.
	
	Just as a remark: The reason we don't take for $p$ the function whose fiber over $n$ is $\{1,\dots, |n|\}$, is because then the usual flattening map $f^\flat : \underset{u \in U}{\coprod} \underset{i \in \{1,\dots, |x(u)|\}}{\coprod} \{1,.\dots,|f(i)| \} \rightarrow \underset{u \in U}{\coprod} \{1,\dots, \depsum{i=1}{x(u)}f(u,i)\}$ would be impossible to define, because if $f$ takes on some negative values, then $j + \depsum{k=1}{i-1}f(u,k)$ might not lie in $\{1,\dots, \depsum{i=1}{x(u)}f(u,i)\}$. If a dependent adder contains negative numbers, one has to add some redundant information in the fibers of $p$.
	\item The reals $\mathbb{R}$ work exactly like the non-negative reals $\bbR$ from Section \ref{reals}, except like for the integers $\mathbb{Z}$ we let $p: F \rightarrow \mathbb{R}$ be the second product projection $\mathbb{R} \times \mathbb{R} \rightarrow \mathbb{R}$.
	\item The \enquote{complete local ring at infinity} $[-1,1]$ works exactly like $\mathbb{R}$ above.
	\item For the complex numbers $\mathbb{C}$ we choose as background category $\mathscr{C}$ the category of complex analytic spaces from \cite{fischer2006complex}. This category has pullbacks by \cite[Corollary 0.32]{fischer2006complex}. In this category a morphism $\mathbb{C} \rightarrow \mathbb{C}$ is an entire holomorphic function $\mathbb{C} \rightarrow \mathbb{C}$.
	Making $\mathbb{C}$ into a dependent adder in this category works very similar to the real numbers $\mathbb{R}$.
\end{enumerate}

Finally a class of degenerate examples not mentioned in the introduction:
If $\mathscr{C}$ is a category with pullbacks and $M$ is a monoid in $\mathscr{C}$, then $A:= M$ can be made into a dependent adder in $\mathscr{C}$. We define $F := M$, $p := id_M : F \rightarrow M$. Then for every $x : U \rightarrow M$ we have $\llbracket x \rrbracket \cong U$.We define a natural function $\mathrm{Hom}_{\mathscr{C}}(U,M) \rightarrow \mathrm{Hom}_{\mathscr{C}}(U,M)$ by sending $y : U \rightarrow M$ to the multiplication $y \cdot x: U \rightarrow M$.
 For every $f : \llbracket x \rrbracket  = U \rightarrow M$ we define
 $f^\flat := id_U : \llbracket f \rrbracket = U \rightarrow U = \llbracket \oversum{x} f \rrbracket$. With this $M$ becomes a dependent adder.

\section{Right Modules over Dependent Adders}\label{sectionRightModules}
Let $\mathscr{C}$ be a category with pullbacks and $A$ a dependent adder in $\mathscr{C}$.

\begin{dfn}
	A \textit{dependent right $A$-module} consists of
\begin{enumerate}
	\item An object $M$ in $\mathscr{C}$
	\item For every $x : U \rightarrow A$ a dependent sum function
	$$ \oversum{x} : \mathrm{Hom}_{\mathscr{C}}( \llbracket x \rrbracket, M) \rightarrow \mathrm{Hom}_{\mathscr{C}}(U,M)$$
	natural in $x \in \mathscr{C}/A$.\\
	For $x : U \rightarrow A$, $f : \llbracket x \rrbracket \rightarrow A$ and $g : \llbracket \oversum{x} f \rrbracket \rightarrow M$ we define $f \boxtimes g := \oversum{f} g \circ f^\flat$.
\end{enumerate}
satisfying the following axioms
\begin{enumerate}
	\item Unit axiom: For $x : U \rightarrow A$ and $f : \llbracket x \rrbracket \rightarrow M$ 
	$$\oversum{const_1} f \circ const_1^\flat = f$$
	\item Sum associativity axiom: For $x : U \rightarrow A$, $f : \llbracket x \rrbracket \rightarrow A$ and $g : \llbracket \oversum{x} f \rrbracket \rightarrow M$
	$$\oversum{\oversum{x} f} g = \oversum{x} \oversum{f} g \circ f^\flat$$
\end{enumerate}
\end{dfn}

While we are at the present moment still hesitant to provide a definition of morphisms of dependent adders, we do not hesitate to define morphisms of dependent right modules.
\begin{dfn}
	Given two dependent right $A$-modules $M$ and $N$, an \textit{$A$-linear map} $M \rightarrow N$ consists of a morphism $\phi : M \rightarrow N$ in $\mathscr{C}$, such that for every $x : U \rightarrow A$ in $\mathscr{C}$ and $f : \llbracket x \rrbracket \rightarrow M$ we have $$\phi \circ (\oversum{x}(f)) =\oversum{x}(\phi \circ f)$$
\end{dfn}

We now provide a few examples of right modules over dependent adders, and the linear maps between them.

\subsection{$\mathbb{F}_1$-modules and pointed sets}
An $\mathbb{F}_1$-dependent right module $M$ is a pointed set. The basepoint is given by the $0$-dependent sum of the unique function $\emptyset \rightarrow M$.

An $\mathbb{F}_1$-linear map is a basepoint-preserving function.

\subsection{$\mathbb{N}$-modules and monoids}

\begin{thm}
	A dependent right $\mathbb{N}$-module $M$ is the same thing as a monoid.
\end{thm}
\begin{proof}
	
	Given any monoid $M$ we obtain an $\mathbb{N}$-dependent module in the following way: 
	We write the binary operation of $M$ by $+$.
	For any $n \in \mathbb{N}$ and function $f  : \{1,\dots, n\} \rightarrow M$ we write $\depsum{i=1}{n} f(i)$ for the sum $f(1) + f(2) + \dots + f(n)$.
	For every $x : U \rightarrow \mathbb{N}$ in $Set$ we define
	$\oversum{x} : \mathrm{Hom}_{Set}(\llbracket x \rrbracket , M) \rightarrow \mathrm{Hom}_{Set}(U, M)$ by
	$$\oversum{x} (f) (u) := \depsum{i=1}{x(u)} f(u,i)$$
	With these dependent sums $M$ is a dependent right $\mathbb{N}$-module.
	
	Conversely, if $M$ is a dependent right $\mathbb{N}$-module, then we can make $M$ into a monoid. 
	In the following we will use list notation for functions $f: \{1,\dots, n\} \rightarrow S$ going from $\{1,\dots, n\}$ into some set $S$.
	Given $n \in \mathbb{N}$ and $n$ elements $s_1, \dots, s_n \in S$ we write $\langle s_1, \dots, s_n \rangle$ for the function $\{1,\dots, n\} \rightarrow S$ sending $i$ to $s_i$
	
	So for example $\langle0\rangle$ is the function $0 : \{1\} \rightarrow \mathbb{N}$ sending $1$ to $0$. We can form the set $\llbracket \langle 0 \rangle \rrbracket = \emptyset$, and have a dependent sum function $\oversum{\langle 0 \rangle } : \mathrm{Hom}_{Set}(\emptyset, M) \rightarrow \mathrm{Hom}_{Set}(\{1\},M)$.
	We have a unique function $\langle \rangle : \emptyset \rightarrow M$.
	and define the unit $0_M \in M$ by defining
	$$\langle 0_M \rangle := \oversum{\langle 0 \rangle } \langle \rangle $$
	This a priori just defines a function $\langle 0_M \rangle: \{1\} \rightarrow M$, but by $0_M$ we of course just mean the unique element of $M$ in the range of that function.
	
	For two elements $a, b \in M$ we define $a + b \in M$ by
	$$\langle a + b \rangle := \oversum{\langle 2 \rangle} \langle a, b \rangle $$
	We now need to show that this satisfies the unit axioms and associativity axiom of a monoid.
	Take some element $a \in M$ and show $0_M + a = a$.
	In $\mathbb{N}$ we have $\langle 1 \rangle = \langle 0 + 1 \rangle = \oversum{\langle2 \rangle} \langle 0,1 \rangle$.
	So we get
	$$\langle a \rangle = \oversum{\langle 1 \rangle} \langle a \rangle = \oversum{\oversum{\langle 2 \rangle} \langle 0, 1 \rangle  } \langle a \rangle = \oversum{\langle 2 \rangle} \oversum{\langle 0, 1 \rangle  } \langle a \rangle \circ \langle 0, 1 \rangle^\flat $$
	
	Let $\iota_1 : \{1\} \rightarrow \{1,2\}$ be the function sending $1$ to $1$, and let $\iota_2 : \{1\} \rightarrow \{1,2\}$ be the function sending $1$ to $2$.
	The naturality of dependent sums implies that
	$(\oversum{\langle 0, 1 \rangle} \langle a \rangle \circ \langle 0, 1 \rangle^\flat) \circ \iota_1 = \oversum{\langle 0, 1 \rangle \circ \iota_1} \langle a \rangle \circ \langle 0, 1 \rangle^\flat \circ (\iota_1 \underset{A}{\times} id_F)  = \oversum{\langle 0 \rangle} \langle \rangle = \langle 0_M \rangle$
	and 
	$(\oversum{\langle 0, 1 \rangle} \langle a \rangle \circ \langle 0, 1 \rangle^\flat) \circ \iota_2 = \oversum{\langle 0, 1 \rangle \circ \iota_2} \langle a \rangle \circ \langle 0, 1 \rangle^\flat \circ (\iota_2 \underset{A}{\times} id_F)  = \oversum{\langle 1 \rangle} \langle a \rangle = \langle a \rangle$
	so in total we get
	$$\langle a \rangle = \oversum{\langle 2 \rangle} \oversum{\langle 0, 1 \rangle  } \langle a \rangle \circ \langle 0, 1 \rangle^\flat = \oversum{\langle 2 \rangle} \langle 0_M, a \rangle =  \langle 0_M + a \rangle $$
	so $M$ satisfies the left unit axiom. By a similar argument it also satisfies the right unit axiom $a + 0_M = a$.
	For associativity we calculate
	$$\oversum{\langle 3 \rangle} \langle a, b, c \rangle = \oversum{\oversum{\langle 2 \rangle} \langle 2, 1 \rangle  } \langle a, b , c \rangle = \oversum{\langle 2 \rangle} \oversum{\langle 2, 1 \rangle} \langle a, b , c \rangle \circ \langle 2, 1 \rangle^\flat = \langle (a + b) + c \rangle$$
	$$\oversum{\langle 3 \rangle} \langle a, b, c \rangle = \oversum{\oversum{\langle 2 \rangle} \langle 1,2 \rangle  } \langle a, b , c \rangle = \oversum{\langle 2 \rangle} \oversum{\langle 1,2 \rangle} \langle a, b , c \rangle \circ \langle 1,2 \rangle^\flat = \langle a + (b + c) \rangle$$
	so $+$ is associative and $M$ is a monoid.
\end{proof}

\begin{thm}
	A $\mathbb{N}$-linear map between dependent $\mathbb{N}$-modules is the same thing as a monoid homomorphism.
\end{thm}
\begin{proof}
	Given a $\mathbb{N}$-linear map $\phi : M \rightarrow N$ we have for every $a, b \in M$ that
	$$\phi(0_M) = \phi(\oversum{\langle 0 \rangle} \langle \rangle) = \oversum{\langle 0 \rangle} \langle \rangle = 0_N  $$
	and
	$$\phi(a + b) = \phi(\oversum{\langle 2 \rangle} \langle a, b \rangle) = \oversum{\langle 2 \rangle} \langle \phi(a), \phi(b) \rangle = \phi(a) + \phi(b) $$
	so $\phi$ is a monoid homomorphism.
	Conversely, if $\phi : M \rightarrow N$ is a monoid homomorphism, then we can show by induction that $\phi(\depsum{i=1}{x} f(i)) = \depsum{i=1}{x} \phi(f(i)) $ and then $\phi$ is a $\mathbb{N}$-linear map.
\end{proof}

\subsection{$\mathbb{R}$-modules and Banach spaces}
Let $\mathbb{R}$ be the dependent adder of all (possibly negative) real numbers.

\begin{lem}
	Giving a dependent right $\mathbb{R}$-module is equivalent to the following data: We have a metrizable space $M$, and for every $x \in \mathbb{R}$ and continuous function $f: \mathbb{R}\rightarrow M$ an element $\underset{0}{\overset{x}{\int}} f(t)dt$ in $M$, such that
	\begin{enumerate}
		\item (Continuity): For every metrizable space $U$, continuous function $x : U \rightarrow \mathbb{R}$ and continuous function $f : U \times \mathbb{R} \rightarrow M$ the function
		$\underset{0}{\overset{x(u)}{\int}} f(u,t) dt$ is continuous in $u$.
		\item (Unit): For every $m \in M$, if $const_m : \mathbb{R} \rightarrow M$ denotes the constant $m$ function then we have
		$$\underset{0}{\overset{1}{\int}} const_m(t) dt = m $$
		\item (Substitution): For every $x \in \mathbb{R}$, continuous function $f : \mathbb{R} \rightarrow M$ and continuously differentiable function $h : \mathbb{R} \rightarrow \mathbb{R}$ with $h(0) = 0$ we have that
		$$ \underset{0}{\overset{h(x)}{\int}} f(t)dt = \underset{0}{\overset{x}{\int}} f(h(t)) \cdot h^\prime(t) dt$$
	\end{enumerate}
\end{lem}
\begin{proof}
	If the Continuity Axiom is satisfied, then the function sending $x : U \rightarrow \mathbb{R}$ and $f : \mathbb{R} \rightarrow M$ to  $\underset{0}{\overset{x(u)}{\int}} f(u,t) dt$ serves as a natural dependent sum function $\oversum{x} f$ for the dependent right $\mathbb{R}$-module $M$.
	Also every function
	$\oversum{x}: \mathrm{Hom}_{Top_{mtr}}( \llbracket x \rrbracket , M ) \rightarrow \mathrm{Hom}_{Top_{mtr}}(U, M) $ that is natural in $x : U \rightarrow \mathbb{R}$
	is necessarily of the above form, because for any $u \in U$ we have a map $\langle u \rangle : 1 \rightarrow U$, and then naturality tells us that
	$(\oversum{x} f) \circ \langle u \rangle = \oversum{x \circ \langle u \rangle} f \circ (\langle u \rangle \underset{A}{\times} id_F)  $, and this implies that the dependent sum functions for all maps of the form $U \rightarrow \mathbb{R}$ are determined by dependent sum functions for maps of the form $1 \rightarrow \mathbb{R}$
	
	The Unit Axiom we stated above is easily seen to be equivalent to the unit axiom of dependent right $\mathbb{R}$-modules.
	The Substitution Axiom stated above is equivalent to the Sum Associativity Axiom of dependent right $\mathbb{R}$-modules, because of the Fundamental Theorem of Calculus.
\end{proof}
 
So in summary, a dependent right $\mathbb{R}$-module is a space in which one can continuously form integrals of functions, and where one has a \enquote{integration by substitution} rule.

If one wants to study dependent right $\mathbb{R}$-modules, it is probably a good idea to only look at modules that have a zero element $0_M \in M$ satisfying$\underset{0}{\overset{0}{\int}} f(t) dt = 0_M$ for all $f : \mathbb{R} \rightarrow M$.
Without such a zero element, dependent right $\mathbb{R}$-modules can be empty or disconnected, while with a zero element they are always path-connected.

For example, every Banach space is a dependent right $\mathbb{R}$-module, with Bochner integrals as dependent sums.

\begin{thm}
	Let $X$ be a Banach space.
	For every continuous function $x : U \rightarrow \mathbb{R}$ and continuous function $f : U \times \mathbb{R} \rightarrow X$ define $\oversum{x}(f) : U \rightarrow X$ by using the Bochner integral
	$$\oversum{x}(f)(u) := \underset{0}{\overset{x(u)}{\int}} f(u,t) dt = \begin{cases}
		\underset{[0,x(u)]}{\int} f(u,t)dt &, x(u) \geq 0\\
		- \underset{[x(u),0]}{\int} f(u,t)dt &, x(u) < 0
	\end{cases}$$
	Then $X$ is a right $\mathbb{R}$-module.
\end{thm}
\begin{proof}
	Take $x : U \rightarrow \mathbb{R}$, $f : \llbracket x \rrbracket \rightarrow X$ and $u \in U$. Assume without loss of generality $x(u) \geq 0$. We need to show that the integral $\underset{[0,x(u)]}{\int} f(u,t) dt$ exists.
	Since $f(u,-) : [0,x(u)] \rightarrow X$ is a continuous function, it is Bochner-measurable.
	The composite function $[0,x(u)] \overset{f(u,-)}{\rightarrow} X \overset{||.||_X}{\rightarrow} \mathbb{R}$ is Lebesgue-integrable, because it is continuous and $[0,x(u)]$ is compact.
	Therefore by Bochner's criterium, the function $f(u,-) : [0,x(u)] \rightarrow X$ is Bochner-integrable, and the integral $\underset{[0,x(u)]}{\int} f(u,t) dt$ exists in $X$.
	
	The Continuity Axiom can be shown exactly like in Lemma \ref{lemmaIntegralContinuous}.
	
	The Unit Axiom $\underset{[0,1]}{\int} x dt = x$ is obvious, because the constant $x$ function is a simple function.
	
	Let us now prove the Substitution Axiom.
	For any differentiable function $f : [0,x] \rightarrow [0,y]$ between real intervals and every differentiable function $g : [0,y] \rightarrow X$ going into a Banach space $X$, we have the chain rule $(g \circ f)^\prime(t) = g^\prime(f(t)) \cdot f^\prime(t)$.
	
	Also for any continuously differentiable function $f : [0,x] \rightarrow X$ going into a Banach space $X$ we have a Fundamental Theorem of Calculus, stating $f(x) - f(0) = \underset{0}{\overset{x}{\int}} f^\prime(t)dt$. This is proven in \cite[Proposition A.2.3]{liu2015stochastic}.
	
	By combining the chain rule with the Fundamental Theorem of Calculus we obtain the substitution rule.
\end{proof}

\begin{thm}
	Let $X, Y$ be Banach spaces.
	Then a continuous linear operator $T: X \rightarrow Y$ is the same thing as 
	an $\mathbb{R}$-linear map $T: X \rightarrow Y$ in the dependent module sense between the corresponding dependent $\mathbb{R}$-modules.
\end{thm}
\begin{proof}
	If $T: X \rightarrow Y$ is a continuous linear operator between Banach spaces, then for every $x \in \mathbb{R}$ and continuous function $f : \mathbb{R} \rightarrow X$ we have that
	$$T\underset{0}{\overset{x}{\int}} f(t)dt = \underset{0}{\overset{x}{\int}} Tf(t)dt $$
	because every Bochner integral is a limit of integrals of simple function, and continuous functions commute with limits, and linear operators commute with integrals of simple functions.
	This implies that $T$ is a $\mathbb{R}$-linear map in the dependent module sense.
	Conversely, take an $\mathbb{R}$-linear map in the dependent module sense $T : X \rightarrow Y$. By definition $T$ is a morphism in $Top_{mtr}$, so $T$ is continuous. We need to show that for all $x, y \in X$ that $T(x+y) = T(x) + T(y)$.
	Define $\phi_{x,y}: [0,1] \rightarrow X$, $\phi(t) := t\cdot y + (1-t)\cdot x$.
	For every $\epsilon >0$ with $\epsilon < 1$ define $\gamma_\epsilon : [0,2] \rightarrow X$ by
	$$\gamma_\epsilon (t) := \begin{cases}
		x&, t < 1 - \epsilon\\
		y&, t > 1 + \epsilon \\
		\phi_{x,y}(\frac{t - 1 + \epsilon}{2\epsilon})&, 1- \epsilon \leq t \leq 1+ \epsilon
	\end{cases} $$
	Then each $\gamma_\epsilon$ is continuous and
	$$T(x+y) = T(\underset{\epsilon \rightarrow 0}{\mathrm{lim}}   \underset{0}{\overset{2}{\int}} \gamma_\epsilon(t)dt  ) = \underset{\epsilon \rightarrow 0}{\mathrm{lim}}   \underset{0}{\overset{2}{\int}} T(\gamma_\epsilon(t))dt = T(x) + T(y) $$
	so $T$ is a linear operator in the usual Banach space sense.
\end{proof}

\subsection{$Cat$-modules and cocomplete categories}

If $M$ is a cocomplete category, then $M$ is a $Cat$-module in the following way:

For every $x : U \rightarrow Cat$ we define a function
$$\oversum{x} : \mathrm{Hom}_{hCAT}(\llbracket x \rrbracket, M) \rightarrow \mathrm{Hom}_{hCAT}(U, M) $$
in the following way:
For every $f : \llbracket x \rrbracket \rightarrow M$ and we define a functor $\oversum{x}(f)^\heartsuit : U \rightarrow M$ by sending $u \in U$ to
$$\oversum{x}(f)^\heartsuit (u) := \underset{i \in x(u)}{\mathrm{colim} } f(u,i) $$
and sending a morphism $\alpha :u  \rightarrow v$ in $U$ to the canonical map
$$ \underset{i \in x(u)}{\mathrm{colim} } f(u,i) \rightarrow \underset{i \in x(v)}{\mathrm{colim} } f(v,i)$$
in $M$.
We define $\oversum{x}(f) := \mathrm{Ho}(\oversum{x} (f)^\heartsuit)$.
Then $M$ satisfies the Unit Axiom, because if we have an object $m \in M$ and consider the diagram $f : 1 \rightarrow M$ sending the unique object of $1$ to $m$, then the colimit of $f$ is isomorphic to $m$ again.
$M$ satisfies the Sum Associativity axiom because for every functor $F: I \rightarrow Cat$ and $G : \oplaxcolim{i \in I} F(i) \rightarrow M$ we have a canonical isomorphism
$$\underset{(i,j) \in \oplaxcolim{i \in I} F(i)}{\mathrm{colim}}  G(i,j) \cong \underset{i \in I}{\mathrm{colim} } \underset{j \in F(i)}{\mathrm{colim} } G(i,j)$$
in $M$. So $M$ is a right $Cat$-module.

If $M$ is a complete category, then $M$ is also a $Cat$-module, because in any such category we have a canonical isomorphism
$$\underset{(i,j) \in \oplaxcolim{i \in I} F(i)}{\mathrm{lim}}  G(i,j) \cong \underset{i \in I}{\mathrm{lim} } \underset{j \in F(i)}{\mathrm{lim} } G(i,j)$$

\begin{thm}
	Let $C$, $D$ be cocomplete categories, and regard $C$, $D$ as $Cat$-modules with colimits as dependent sums.
	Then a functor $F : C \rightarrow D$ is a $Cat$-linear map if and only if $F$ preserves small colimits, in the usual sense of sending small colimit cocones to small colimit cocones.
	
\end{thm}
\begin{proof}
	Let $F : C \rightarrow D$ be a $Cat$-linear map.
	Let $G : I \rightarrow C$ be a functor.
	It is immediately clear that $F(\underset{i \in I}{\mathrm{colim}} G(i) ) \cong \underset{i \in I}{\mathrm{colim}} F(G(i))$. But for $F$ to preserve colimits it does not just need to send colimit objects to colimit objects, but it needs to preserve colimit cocones.
	Let $U = (\cdot \rightarrow \cdot)$ be the category with two objects and one morphism between them. Let $x : U \rightarrow Cat$ be the constant functor sending everything to $I \in Cat$. Then $\llbracket x \rrbracket \cong I \times U$.
	To define a functor $g : I \times U \rightarrow Cat$ we need to define two functors $g_0, g_1 : I \rightarrow Cat$ and a natural transformation $\tau : g_0 \rightarrow g_1$.
	Let $g_0 := G$. Let $g_1$ be the constant functor sending every $i \in I$ to $\underset{j \in I}{\mathrm{colim}} G(j)$.
	Let $\tau : g_0 \rightarrow g_1$ be the colimiting cocone of $G$. So for every $i \in I$ the map $\tau_i : G(i) \rightarrow \underset{j \in I}{\mathrm{colim}} G(j)$ is the canonical colimit inclusion map.
	
	Since $F$ is a $Cat$-linear map we have
	$$F \circ \oversum{x} g = \oversum{x} F \circ g $$
	
	This means we have a commutative diagram
	$$\xymatrix{ F( \underset{i \in I}{\mathrm{colim}} G(i)  )\ar[d]  \ar[r]^\sim & \underset{i \in I}{\mathrm{colim}} F(G(i)) \ar[d] \\
	   F( \underset{i \in I}{\mathrm{colim}} g_1(i)  ) \ar[r] & \underset{i \in I}{\mathrm{colim}} F(g_1(i))   } $$
	Now $g_1$ is a constant functor, so $ \underset{i \in I}{\mathrm{colim}} g_1(i) =  \underset{i \in I}{\mathrm{colim}} G(i)$ and $ \underset{i \in I}{\mathrm{colim}} F(g_1(i)) = F( \underset{i \in I}{\mathrm{colim}} G(i)) $.
	The above diagram then implies that $F$ preserves not just the colimit object but the whole colimiting cocone.
\end{proof}

Similarly, if $C$, $D$ are complete categories regarded as $Cat$-modules with limits as dependent sums, then a functor $F : C \rightarrow D$ is a $Cat$-linear map if $F$ preserves small limits.

\section{Left Modules over Dependent Adders}\label{sectionLeftModules}

\begin{dfn}
	
	Given a category with pullbacks $\mathscr{C}$ and a dependent adder $A$, an \textit{$A$-dependent left module} consists of
	\begin{enumerate}
		\item An object $M \in \mathscr{C}$
		\item A morphism $p : F_M \rightarrow M$ in $\mathscr{C}$.
		For every $m : U \rightarrow M$ we write $\llbracket m \rrbracket_M$ for the pullback $U \underset{M}{\times} F_M$.
		\item For every $m: U \rightarrow M$ a function of sets
		$$\oversum{m} : \mathrm{Hom}_{\mathscr{C}}(\llbracket m \rrbracket_M, A ) \rightarrow \mathrm{Hom}_{\mathscr{C}}(U,M) $$
		natural in $m \in \mathscr{C}/M$.
		\item For every $m : U \rightarrow M$ and $f : \llbracket m \rrbracket_M \rightarrow A$ a flattening function
		$$f^{M,\flat} : \llbracket f \rrbracket \rightarrow \llbracket \oversum{m} f \rrbracket_M $$
		over $U$.
		For $m : U \rightarrow M$, $f : \llbracket m \rrbracket_M \rightarrow A$ and $g : \llbracket \oversum{m} f \rrbracket_M \rightarrow A$  we define $f \boxtimes g := \oversum{f} g \circ f^{M,\flat}$.
	\end{enumerate}
	such that the following axioms are satisfied
	\begin{enumerate}
		\item Unit Axiom: For $m : U \rightarrow M$ and $const_1 : \llbracket m \rrbracket_M \rightarrow A$ the constant $1_A$ function, we demand
		$$\oversum{m} const_1 = m$$
		\item Sum Associativity Axiom: For $m : U \rightarrow M$, $f : \llbracket m \rrbracket_M \rightarrow A$ and $g : \llbracket \oversum{m} f \rrbracket_M \rightarrow A$ we demand that
		$$ \oversum{\oversum{m}f}g = \oversum{m} \oversum{f} g \circ f^{M,\flat} $$
		\item Flatten Associativity Axiom For $m : U \rightarrow M$, $f : \llbracket m \rrbracket_M \rightarrow A$ and $g : \llbracket \oversum{m} f \rrbracket_M \rightarrow A$ the following diagram commutes
		$$\xymatrix{
			\llbracket g \circ f^{M,\flat} \rrbracket \ar[d]_{ f^{M,\flat} \underset{A}{\times} id_F  } \ar[r]^(.55){(g \circ f^{M,\flat})^\flat} & \llbracket f \boxtimes g \rrbracket \ar[d]^{(f \boxtimes g)^{M,\flat}  } \\
			\llbracket g \rrbracket \ar[r]^(.45){g^{M,\flat}} & \llbracket \oversum{x} f \boxtimes g \rrbracket 
		} $$
	\end{enumerate}
	
\end{dfn}

The interval $[0,n]$ has both a right and a left $[0,1]$-dependent module structure. The right module structure comes from the fact that for any $x \in [0,1]$ and continuous map $f : [0,x] \rightarrow [0,n]$ we have $\underset{0}{\overset{x}{\int}} f(t)dt \in [0,n]$.
The left module structure comes from the fact that for any $x \in [0,n]$ and continuous map $f : [0,x] \rightarrow [0,1]$ we have $\underset{0}{\overset{x}{\int}} f(t)dt \in [0,n] $.

\subsection{$Top_{open}$ as a left $Set$-module using étalé spaces of presheaves}

The category of topological spaces and open maps $Top_{open}$ has, up to isomorphism, the structure of a left $Set$-dependent module. For any topological space $X \in Top_{open}$ we define an $X$-indexed family of sets to be a presheaf $\mathscr{F}$ on $X$, and define the sum of such a sheaf to be its étalé space $Et(\mathscr{F})$.

Let $\mathscr{C} = hCAT$ be the homotopy category of small categories from Section \ref{categories}.
Let $A := Set$ be the category of very small sets,
$Set$ is a dependent adder in $hCAT$, quite similarly to how $Cat$ is a dependent adder in $hCAT$.
For any $x : U \rightarrow Set$ we have $\llbracket x \rrbracket \cong \oplaxcolim{u \in U} x(u)$, where $x(u)$ is regarded as a discrete category.
The dependent sums of $Set$ are the coproducts: Given $f : \llbracket x \rrbracket \rightarrow Set$ we define
$\oversum{x}f := \mathrm{Ho}( \lambda u. \underset{t \in x(u)}{\coprod} f(u,t)  )$.
The flattening map of $f^\flat$ of $f$ is given by taking $\mathrm{Ho}$ of the natural isomorphism
$$ \oplaxcolim{(u,t) \in \oplaxcolim{u \in U} x(u)} f(u,t) \rightarrow \oplaxcolim{u \in U} \underset{t \in x(u)}{\coprod} f(u,t)$$
where one needs to note that $\underset{t \in x(u)}{\coprod} f(u,t)$ is in fact the same thing as $\oplaxcolim{t \in x(u)} f(u,t)$, because oplax colimits over discrete categories are just coproducts.
With this $Set$ is a dependent adder in $hCAT$.

Now let $M := Top_{open}$ the category of very small topological spaces with open maps.

We quickly recall the construction of the étalé space $Et(\mathscr{F})$ of a presheaf $\mathscr{F}$ on a topological space $X$. See \cite[Section II.6]{maclane1994sheaves} for a reference. For each $x \in X$ let $\mathscr{F}_x$ be the stalk of $\mathscr{F}$ at x. The underlying set of $Et(\mathscr{F})$ is the coproduct of all the stalks of $\mathscr{F}$.
$$ Et(\mathscr{F}) = \underset{x \in X}{\coprod} \mathscr{F}_x$$
For every open subset $U \subseteq X$ and section $s \in \mathscr{F}(U)$, we have for every $x \in U$ an element $s_x \in \mathscr{F}_x$ and can define the set $\epsilon_{\mathscr{F}}(U,s) := \{ (x,s_x) | x \in U  \} \subseteq Et(\mathscr{F}) $. We put on $Et(\mathscr{F})$ the topology generated by the sets $\epsilon_{\mathscr{F}}(U,s)$ for all open $U \subseteq X$ and $s \in \mathscr{F}(U)$. The notation $\epsilon_{\mathscr{F}}(U,s)$ will be used below a few times.

Étalé spaces are functorial: If $\mathscr{F}, \mathscr{G}$ are presheaves on $X$ and $ \tau: \mathscr{F} \rightarrow \mathscr{G}$ is a morphism of presheaves, then for every $x \in X$ we get a map on stalks $\tau_x : \mathscr{F}_x \rightarrow \mathscr{G}_x$, and these assemble together into a map on étalé spaces $Et(\tau): Et(\mathscr{F}) \rightarrow Et(\mathscr{G})$. This map is always an open map, because $Et(\tau)(\epsilon_\mathscr{F}(U,s)) = \epsilon_\mathscr{G}(U, \tau_U(s))$. With this construction $Et : \mathrm{Psh}(X) \rightarrow Top_{open}$ is a functor from the category of presheaves on $X$ to the category of topological spaces with open maps.

For any open map $f : X \rightarrow Y$ in $Top_{open}$ and every presheaf $\mathscr{F}$ on $Y$ define a presheaf $f^*(\mathscr{F})$ on $X$ by sending $U \in \Ouv(X)$ to $f^*(\mathscr{F})(U) := \mathscr{F}(f(U))$.
\begin{lem} \label{lemmaEtFunctorial2}
	For every open map $f : X \rightarrow Y$ in $Top_{open}$ and every presheaf $\mathscr{F}$ on $Y$ there is a natural open map
	$$ Et_f :  Et(f^*(\mathscr{F})) \rightarrow Et(\mathscr{F})$$ in $Top_{open}$
\end{lem}
\begin{proof}
	For every $x \in X$ we have a map on stalks $$f^*(\mathscr{F})_x = \underset{x \in U \in \Ouv(X)}{\mathrm{colim}} \mathscr{F}(f(U)) \rightarrow  \underset{x \in V \in \Ouv(Y)}{\mathrm{colim}} \mathscr{F}(V) = \mathscr{F}_{f(x)}$$ because for every open neighborhodd $U$ of $x$ in $X$, $f(U)$ is an open neighborhood of $f(x)$ in $Y$. These maps assemble together into a map
	$$Et_f: Et(f^*(\mathscr{F})) = \underset{x \in X}{\coprod} f^*(\mathscr{F})_x \rightarrow \underset{y \in Y}{\coprod} \mathscr{F}_y = Et(\mathscr{F})  $$
	which is open because $Et_f(\epsilon_{f^*(\mathscr{F})} (U,s)) = \epsilon_{\mathscr{F}}(f(U), s)$.
\end{proof}

Let us now explain how étalé spaces make $Top_{open}$ into a left $Set$-module.

Consider the functor $\Ouv : Top_{open} \rightarrow CAT$, sending $X \in Top$ to the poset $\Ouv(X)^{op}$ of open subsets of $X$ with reverse inclusions as morphisms, and sending an open map $f : X \rightarrow Y$ to the image functor $\Ouv(f) : \Ouv(X)^{op} \rightarrow \Ouv(Y)^{op}$.
The functor $\Ouv$ gives rise to a Grothendieck op-fibration $p^\heartsuit : \oplaxcolim{X \in Top_{open}} \Ouv(X)^{op} \rightarrow Top_{open}$.\\
Define $F_{M} := \oplaxcolim{X \in Top_{open}} \Ouv(X)^{op}$ and $p_{M} := \mathrm{Ho}(p^\heartsuit)$.

For any $X \in Top_{open}$ the fiber of $p_{M}$ over $X$ is $\Ouv(X)^{op}$, and a morphism $\Ouv(X)^{op} \rightarrow Set$ in $hCAT$ is an isomorphism class of presheaves on $X$.

More generally, for any functor $x : U \rightarrow Top_{open}$ in $hCAT$ we have
$$\llbracket x \rrbracket_{M} \cong \oplaxcolim{u \in U}  \Ouv(x(u))^{op} $$

For any functor $x : U \rightarrow Top_{open}$ in $hCAT$ we need to define a map
$$\oversum{x} : \mathrm{Hom}_{hCAT}(\llbracket x \rrbracket_{M}, Set) \rightarrow \mathrm{Hom}_{hCAT}(U, Top_{open})$$

For every $f : \llbracket x \rrbracket_{M} \rightarrow Set$ in $hCAT$ we define a functor $(\oversum{x} f)^\heartsuit : U \rightarrow Top_{open}$ the following way:
For any object $u \in U$ we have an inclusion functor $\iota_u: \Ouv(x(u))^{op} \rightarrow \oplaxcolim{v \in U}  \Ouv(x(v))^{op} \cong \llbracket x \rrbracket_{M}$.
Then $f \circ \iota_u : \Ouv(x(u))^{op} \rightarrow Set$ is a presheaf on $x(u)$.
We define $ \oversum{x}(f)^\heartsuit(u)$ to be the étalé space of this presheaf.
$$ \oversum{x}(f)^\heartsuit(u) := Et(f \circ \iota_u )$$

Given a morphism $\alpha : u \rightarrow v$ in $U$ we get a natural transformation $\iota_u \Rightarrow \iota_v \circ \Ouv(x(\alpha))$, where $x(\alpha) : \Ouv(x(u))^{op} \rightarrow \Ouv(x(v))^{op}$ is the image functor associated to the open map $x(\alpha)$. This then induces a morphism of presheaves $\tau_\alpha: f \circ \iota_u  \rightarrow f \circ \iota_v \circ \Ouv(x(\alpha)) $, which then induces an open map on étalé spaces $Et(\tau_\alpha) : Et( f \circ \iota_u ) \rightarrow Et(f \circ \iota_v \circ \Ouv(x(\alpha)))$.
By Lemma \ref{lemmaEtFunctorial2} we have a canonical open map $Et(f \circ \iota_v \circ \Ouv(x(\alpha))) \rightarrow Et(f \circ \iota_v)$. So in total we obtain an open map $$ \oversum{x}(f)^\heartsuit(u) = Et(f \circ \iota_u ) \rightarrow  Et(f \circ \iota_v ) =  \oversum{x}(f)(v)  $$
so $ \oversum{x}(f)^\heartsuit : U \rightarrow Top_{open}$ is a functor. 

Now define $\oversum{x}(f) := \mathrm{Ho}( \oversum{x}(f)^\heartsuit)$.
Then $\oversum{x}$ is a natural function $$\oversum{x} : \mathrm{Hom}_{hCAT}(\llbracket x \rrbracket_{M}, Set) \rightarrow \mathrm{Hom}_{hCAT}(U, Top_{open})$$

Next we need to define for any small category $D$, functor $x : D \rightarrow Top_{open}$ and functor $f : \llbracket x \rrbracket_M \rightarrow Set$ the flattening map 
$$f^{M,\flat} : \oplaxcolim{d \in D} \oplaxcolim{ U \in \Ouv(x(d))^{op}  } f(d,U) \rightarrow \oplaxcolim{d \in D} \Ouv( Et(f \circ \iota_d ))^{op}  $$

where the $f(d,U)$ are regarded as discrete categories.
Just to make the notation a bit more intuitive, define for every $d \in D$ that $X_d := x(d)$ and $\mathscr{F}_d := f \circ \iota_d$. Then $X_d$ is a topological space and $\mathscr{F}_d$ is a presheaf on $X_d$, and we have to define a map
$$f^{M,\flat} : \oplaxcolim{d \in D} \oplaxcolim{ U \in \Ouv(X_d)^{op}  } \mathscr{F}_d(U) \rightarrow \oplaxcolim{d \in D} \Ouv( Et(\mathscr{F}_d ))^{op}  $$

We have for every $d \in D$ a map $\epsilon_{\mathscr{F}_d} : \oplaxcolim{U \in \Ouv(X_d)^{op}} \mathscr{F}_d(U) \rightarrow \Ouv(Et(\mathscr{F}_d))^{op}$ sending an open subset $U \subseteq X_d$ and a section $s \in \mathscr{F}_d(U)$ to the open subset $\epsilon_{\mathscr{F}_d}(U,s) = \{(x,s_x) | x \in X_d  \}$ of the étalé space $Et(\mathscr{F}_d)$. If we have a morphism $\alpha : (U,s) \rightarrow (V,t)$ in $\oplaxcolim{U \in \Ouv(X_d)^{op}} \mathscr{F}_d(U)$ then we have an inclusion $U \supseteq V$ of open subsets and a morphism $s|_V \rightarrow t$ in the discrete category $\mathscr{F}_d(V)$. So we have in fact an identity $s|_V = t$. This implies an inclusion of subsets $\epsilon_{\mathscr{F}_d}(U,s) \supseteq \epsilon_{\mathscr{F}_d}(V,t)$ in $Et(\mathscr{F}_d)$.
So $\epsilon_{\mathscr{F}_d} : \oplaxcolim{U \in \Ouv(X_d)^{op}} \mathscr{F}_d(U) \rightarrow \Ouv(Et(\mathscr{F}_d))^{op}$ is in fact a functor.
One can now check that this functor is natural in $d$. This naturality claim comes down to the assertion that for any morphism $\alpha : d \rightarrow e$ in $D$ and $U \in \Ouv(X_d)^{op}$ and $s \in \mathscr{F}_d(U)$ we have
$$\oversum{x}(f)^\heartsuit(\alpha) (\epsilon_{\mathscr{F}_d}(U,s)) = \epsilon_{\mathscr{F}_e}( x(\alpha)(U), f(\alpha,id_U)(s)  )$$
and this assertion is in fact true. With this naturality, we can then use the universal property of the oplax colimit to get a functor $f^{M,\flat,\heartsuit}$ that satisfies $f^{M,\flat,\heartsuit} \circ \iota_d = \epsilon_{\mathscr{F}_d}$.
We can then define $f^{M,\flat} := \mathrm{Ho}(f^{M,\flat,\heartsuit})$ and then have a flattening function for our left $Set$-module $Top_{open}$.

We now need to verify the axioms of a left $Set$-module.

Unit Axiom:
Let $X$ be a topological space, and $\mathscr{F}$ the constant $1$ presheaf $\mathscr{F}(U) = 1$.
Then for every $x \in X$ we also have $\mathscr{F}_x \cong 1$, and using this we obtain a homeomorphism $Et(\mathscr{F}) \cong X$ in $Top_{open}$. Since this homeomorphism is natural in $X$, this implies the Unit Axiom for the left $Set$-module $Top_{open}$.

Sum Associativity Axiom:
Take a topological space $X$, a presheaf $\mathscr{F}$ on $X$ and a presheaf $\mathscr{G}$ on $Et(\mathscr{F})$.

To show the Sum Associativity Axiom we need to show that there is a homeomorphism
$$\oversum{\oversum{X}\mathscr{F}}\mathscr{G} \cong \oversum{X} \oversum{\mathscr{F}} \mathscr{G} \circ \mathscr{F}^{M,\flat} $$
natural in $X$, $\mathscr{F}$ and $\mathscr{G}$.

The space $\oversum{\oversum{X}\mathscr{F}}\mathscr{G}$ is homeomorphic the étalé space $Et(\mathscr{G})$ of $\mathscr{G}$.

The map $\mathscr{F}^{M,\flat}$ is the functor $\epsilon_{\mathscr{F}} : \oplaxcolim{V \in \Ouv(X)^{op} } \mathscr{F}(V) \rightarrow \Ouv(Et(\mathscr{F}))^{op} $ sending $(V,s)$ to $\epsilon_{\mathscr{F}}(V,s) = \{(x,s_x) | x \in V \}$.
	
Let $\mathscr{F} \boxtimes \mathscr{G} := \oversum{\mathscr{F}} \mathscr{G} \circ \mathscr{F}^{M,\flat}$. Then $\mathscr{F} \boxtimes \mathscr{G}$ is isomorphic to the presheaf on $X$ that sends an open subset $V \subseteq X$ to $\underset{s \in \mathscr{F}(V)}{\coprod} \mathscr{G}(\epsilon_{\mathscr{F}}(V,s))$. $$(\mathscr{F} \boxtimes \mathscr{G})(V) \cong \underset{s \in \mathscr{F}(V)}{\coprod} \mathscr{G}(\epsilon_{\mathscr{F}}(V,s))$$

To prove the Sum Associativity Axiom we now need to show that there is a natural homeomorphism $$Et(\mathscr{G}) \cong Et(\mathscr{F} \boxtimes \mathscr{G}) $$

\begin{lem}
	For every $x \in X$ there is a natural isomorphism
	$$(\mathscr{F} \boxtimes \mathscr{G})_x \cong \underset{s \in \mathscr{F}_x  }{\coprod} \mathscr{G}_{(x,s)}  $$
	where $\mathscr{G}_{(x,s)}$ means the stalk of $\mathscr{G}$ at $(x,s) \in \underset{x \in X}{\coprod} \mathscr{F}_x = Et(\mathscr{F})$.
\end{lem}
\begin{proof}
	We have
	$(\mathscr{F} \boxtimes \mathscr{G})_x \cong \underset{x \in U \in \Ouv(X)}{\mathrm{colim}} \underset{s \in \mathscr{F}(U)}{\coprod}\mathscr{G}(\epsilon_{\mathscr{F}}(U,s))$.
	Now for every open neighborhood $U$ of $x$ and every $s \in \mathscr{F}(U)$ we have a map $\mathscr{F}(U) \rightarrow \mathscr{F}_x$, and we know that $\epsilon_{\mathscr{F}}(U,s)$ is an open neighborhood of $(x,s) \in Et(\mathscr{F})$ so we have a map $\mathscr{G}(\epsilon_{\mathscr{F}}(U,s)) \rightarrow \mathscr{G}_{(x,s)}$.
	These maps assemble together into a map $(\mathscr{F} \boxtimes \mathscr{G})_x \rightarrow \underset{s \in \mathscr{F}_x  }{\coprod} \mathscr{G}_{(x,s)} $.
	We claim that this map is surjective:
	If we have $(s,t) \in \underset{s \in \mathscr{F}_x  }{\coprod} \mathscr{G}_{(x,s)}$, then there exists an open neighborhood $V$ of $(x,s)$ in $Et(\mathscr{F})$ and a section $\tilde{t} \in \mathscr{G}(V)$ such that $t = \tilde{t}_{(x,s)}$. Since the topology of $Et(\mathscr{F})$ is generated by open sets of the form $\epsilon_{\mathscr{F}}(W,\tilde{s})$, we know there exists an open set $W \subseteq X$ and some $\tilde{s} \in \mathscr{F}(W)$ such that $(x,s) \in \epsilon_{\mathscr{F}}(W,\tilde{s}) \subseteq V$. Then $(\tilde{s}, t|_{\epsilon_{\mathscr{F}}(W,\tilde{s})} )$ lies in $\underset{y \in \mathscr{F}(W)}{\coprod}\mathscr{G}(\epsilon_{\mathscr{F}}(W,y)) \cong (\mathscr{F} \boxtimes \mathscr{G})(W)$ and $(\tilde{s}, t|_{\epsilon_{\mathscr{F}}(W,\tilde{s})} )_x = (s,t)$. One can similarly check that the above map is injective, and then it is an isomorphism of sets.
\end{proof}

With this lemma we get a natural bijective map $\Phi: Et(\mathscr{G}) \rightarrow Et(\mathscr{F} \boxtimes \mathscr{G})$ defined by
\footnotesize
$$Et(\mathscr{G}) = \underset{y \in Et(\mathscr{F})}{\coprod} \mathscr{G}_y = \underset{(x,s) \in \underset{x \in X}{\coprod} \mathscr{F}_x }{\coprod} \mathscr{G}_{(x,s)} \cong \underset{x \in X}{\coprod} \underset{s \in \mathscr{F}_x}{\coprod} \mathscr{G}_{(x,s)} \cong \underset{x \in X}{\coprod} (\mathscr{F} \boxtimes \mathscr{G})_x = Et(\mathscr{F} \boxtimes \mathscr{G}) $$
\normalsize

We now just need to show that this map is open and continuous.

\begin{lem}\label{lemmaOpenSetsInEt}
	For all open $U \subseteq X$, $s \in \mathscr{F}(U)$ and $t \in \mathscr{G}(\epsilon_{\mathscr{F}}(U,s))$ we have
	$$\epsilon_{\mathscr{F} \boxtimes \mathscr{G}}(U, (s,t)  ) = \Phi(\epsilon_{\mathscr{G}}(\epsilon_{\mathscr{F}}(U,s), t)) $$
	where $\Phi$ is the natural bijective map $\Phi: Et(\mathscr{G}) \rightarrow Et(\mathscr{F} \boxtimes \mathscr{G})$.
\end{lem}
\begin{proof}
    $$\Phi(  \epsilon_{\mathscr{G}}(\epsilon_{\mathscr{F}}(U,s), t))  = \Phi(\{ (y,t_y) | y \in \epsilon_{\mathscr{F}}(U,s)\}) =$$ $$ =\Phi(\{ (y,t_y) | y = (x,s_x), x \in U\}) = \{ (x, (s,t)_x) |  x \in U  \} = \epsilon_{\mathscr{F} \boxtimes \mathscr{G}}(U, (s,t)  ) $$
\end{proof}

This lemma then shows that our map $Et(\mathscr{G}) \rightarrow Et(\mathscr{F} \boxtimes \mathscr{G}) $ is a homeomorphism. This homeomorphism is natural in $X$ , $\mathscr{F}$ and $\mathscr{G}$ and this then proves the Sum Associativity Axiom of $Top_{open}$.

Flatten Associativity Axiom:
Take functors $x : D \rightarrow Top_{open}$, $f : \llbracket x \rrbracket_M\rightarrow Set$ and $g : \llbracket \oversum{x} f \rrbracket \rightarrow Set$.

Write $X_d := x(d)$, $\mathscr{F}_d := f(d,-)$, $\mathscr{G}_d := g(d,-)$.
We need to show that the following diagram commutes:
\footnotesize
$$\xymatrix{ \oplaxcolim{d \in D} \oplaxcolim{(U,s) \in \oplaxcolim{U \in \Ouv(X_d)^{op} }  \mathscr{F}_d(U)}  \mathscr{G}_d(\epsilon_{\mathscr{F}_d}(U,s)) \ar[r] \ar[d] & \oplaxcolim{d \in D} \oplaxcolim{U \in \Ouv(X_d)^{op}} (\mathscr{F}_d \boxtimes \mathscr{G}_d)(U)  \ar[d]\\
\oplaxcolim{d \in D} \oplaxcolim{V \in \Ouv(Et(\mathscr{F}_d))^{op}} \mathscr{G}_d(V) \ar[r] & \oplaxcolim{d \in D} \Ouv(Et(\mathscr{G}_d))^{op}  }  $$
\normalsize

Since $\Ouv(Et(\mathscr{G}_d))^{op}$ is a pre-order, any two morphisms with the same domain and codomain coincide in it. For this reason the commutativity of the above diagram can be \enquote{checked on objects}, in the sense that it commutes if and only if for every object $d \in D$ the following diagram commutes

$$\xymatrix{ \oplaxcolim{(U,s) \in \oplaxcolim{U \in \Ouv(X_d)^{op} }  \mathscr{F}_d(U)}  \mathscr{G}_d(\epsilon_{\mathscr{F}_d}(U,s)) \ar[r] \ar[d]^(.6){\epsilon_{\mathscr{F}_d} \underset{A}{\times} id_F} & \oplaxcolim{U \in \Ouv(X_d)^{op}} (\mathscr{F}_d \boxtimes \mathscr{G}_d)(U)  \ar[d]^{\epsilon_{\mathscr{F}_d \boxtimes \mathscr{G}_d} } \\
	 \oplaxcolim{V \in \Ouv(Et(\mathscr{F}_d))^{op}} \mathscr{G}_d(V) \ar[r]^{\epsilon_{\mathscr{G}_d}} &  \Ouv(Et(\mathscr{G}_d))^{op}  }  $$
 
And this follows from Lemma \ref{lemmaOpenSetsInEt}.
So $Top_{open}$ is a left $Set$-module.

\bibliographystyle{plain}

\footnotesize
\textit{Email address}: \texttt{peterjbonart@gmail.com}

\end{document}